\documentclass{siamart}

\usepackage{lipsum}
\usepackage{amsfonts}
\usepackage{graphicx}
\usepackage{epstopdf}

\theoremstyle{definition}
\newtheorem{example}[theorem]{Example}

\theoremstyle{remark}
\newtheorem{remark}[theorem]{Remark}
\usepackage{algorithmic}
\ifpdf
  \DeclareGraphicsExtensions{.eps,.pdf,.png,.jpg}
\else
  \DeclareGraphicsExtensions{.eps}
\fi

\newcommand{\TheTitle}{Parametric Interpolation Framework for Scalar Conservation Laws} 
\newcommand{\TheAuthors}{G. McGregor, J.-C. Nave}

\headers{\TheTitle}{\TheAuthors}

\title{{Parametric Interpolation Framework for Scalar Conservation Laws}}

\author{
  Geoffrey McGregor\thanks{Department of Mathematics, McGill University, Montreal, QC, Canada
    (\email{Geoffrey.McGregor@mail.mcgill.ca}}. The research of GMc was supported in part by the Schulich Scholarship and the Murata Fellowship at McGill University.
  \and
  Jean-Christophe Nave\thanks{Department of Mathematics, McGill University, Montreal, QC, Canada
    (\email{jcnave@math.mcgill.ca} The research of JCN was supported in part by the NSERC Canada Discovery Grants Program.}
}

\usepackage{amsopn}
\usepackage{amsmath}
\usepackage{amsfonts}
\usepackage{graphicx}
\usepackage{mathtools}

\ifpdf
\hypersetup{
  pdftitle={\TheTitle},
  pdfauthor={\TheAuthors}
}
\fi

\begin{document}

\maketitle

\begin{abstract}
In this paper we present a novel framework for obtaining high-order numerical methods for scalar conservation laws in one-space dimension for both the homogeneous and non-homogeneous case. The numerical schemes for these two settings are somewhat different in the presence of shocks, however at their core they both rely heavily on the solution curve being represented parametrically. By utilizing high-order parametric interpolation techniques we succeed to obtain fifth order accuracy ( in space ) everywhere in the computation domain, including the shock location itself. In the presence of source terms a slight modification is required, yet the spatial order is maintained but with an additional temporal error appearing. We provide a detailed discussion of a sample scheme for non-homogeneous problems which obtains fifth order in space and fourth order in time even in the presence of shocks. 
\end{abstract}



\section{Introduction}\label{Intro}
In this paper we consider the 1-D scalar conservation law,\\
\begin{equation}
\begin{cases}
u_t+(F(u))_x=Q(u,x,t)\label{Cauchy}\\
u(x,0)=g(x),
\end{cases}
\end{equation}
where $g$ is piecewise smooth and both $F$ and $Q$ are smooth functions in their respective domains. Here, we focus on the case where $F$ is uniformly convex. 

When seeking numerical solutions to (\ref{Cauchy}), it remains a great difficulty to appropriately handle the formation and propagation of jump discontinuities (shocks). To this end, our goal in the present paper is to develop a numerical framework for solving (\ref{Cauchy}) to high-order such that shocks are located to the same order (or better).

Developing convergent and stable numerical methods for (\ref{Cauchy}) has been an ongoing area of research for well over half of a century. Typically, methods focus on obtaining high-order accuracy in smooth regions of the solution while preserving desirable properties, such as boundedness and sharp features, near discontinuities.  One common approach to this, as discussed in \cite{LeVeque,LevFinite}, is to allow the numerical technique used to vary within the computation domain depending on some numerically approximated smoothness. Finite difference methods can utilize slope-limiters to switch between a high-accuracy scheme in smooth regions and a stable scheme near shocks. Other methods, such as ENO \cite{ENO} and WENO \cite{WENO} schemes, construct a local stencil or weights based on the local data. The resulting schemes are high-order in smooth regions of the solution and avoid spurious oscillations near discontinuities. These schemes form an important part of the current research within numerical conservation laws, for example see \cite{fu2018new,wang2018improved,du2018hermite,hu20186th,shu2016high,qiu2004hermite,qiu2005hermite,luo2007hermite}. Similarly, finite volume methods, which have the added feature of being exactly conservative, utilize flux-limiters, for example see \cite{FinVolLimiter,LevFinite}, to switch between a high-accuracy numerical flux, and a stable numerical flux with lesser accuracy. Overall the focus near shocks is to keep the numerical solution stable, or bounding the variation, to obtain a convergent numerical scheme. These methods accomplish their goal, and are applicable to systems of hyperbolic conservation laws and higher dimensions, although they have difficulty tracking shock position with high-order accuracy. This is a direct result of requiring low-order schemes at the shock location to prevent spurious oscillations and other instabilities. Front tracking \cite{FrontTrack} and Godunov methods \cite{Godunov} explicitly use the Rankine-Hugoniot condition to obtain the correct shock speed, however, generalizations of these methods for predicting shock position to high-order proves to be difficult.  One advantage of these methods is their applicability to systems of conservations laws, however, Godunov type methods can be computationally expensive when resolving Riemann problems with complex wave structure. To overcome this, approximate Riemann solvers such as the HLL method \cite{HLL} and HLLC method \cite{HLLC} were developed to address the computational complexity issue. These methods are particularly useful for solving real-world problems in hyperbolic conservation laws, for example see \cite{balsara2015three,vides2015simple,goetz2018family}. 

Thus far we have focused on the homogeneous case of (\ref{Cauchy}). The addition of source terms adds a considerable challenge from a numerical perspective and the resulting schemes are often quite different from their homogeneous counterpart.  For some examples, see \cite{tseng2004improved,burguete2001efficient,xing2006high,wang2012high}.  One way to handle source terms is to perform a splitting of the horizontal motion, determined by the flux function $F$, and vertical motion, given by the source term $Q$. For example in \cite{Seibold}, the authors solve the homogeneous problem by flowing particles under the method of characteristics and utilizing an area-preserving linear interpolation when merging particles. In the presence of source terms, the characteristic flow and interpolation step ignore the source term, and then once completed the particles are moved vertically according to the source term. The resulting method predicts the shock position to first order in time.

In the present paper we are concerned with finding a numerical framework for solving (\ref{Cauchy}) to high spatial and temporal order while guaranteeing the shock position is determined to high-order as well. While the majority of current research is focused on systems and higher dimensions, methods able to track shock location to high-order are lacking, even in the 1-D scalar case. We therefore focus our efforts on the scalar case first to lay the groundwork for future research in high-order shock location methods for systems and higher dimensions. To achieve our goal, we rely heavily on the characteristic curves associated with (\ref{Cauchy}).  In general, it is not possible to use polynomial interpolation for the characteristic curves as they become multivalued when shocks form, leading to an ill-defined polynomial interpolation problem.  One possible alternative is parametric polynomial interpolation which enables the representation of a larger class of curves in the plane. Before getting into the specifics of parametric interpolation, we first discuss how a parametric representation of the characteristic curves may be used to obtain physically relevant weak solutions of (\ref{Cauchy}). 

The method of characteristics yields the characteristic equations
\begin{align}
\dot{x}&=F'(u) \label{CharEq}\\
\dot{u}&=Q(u,x,t), \nonumber
\end{align}
which has solution parametrized by $x_0$ given by $\langle x(x_0,t) \, ,\, u(x_0,t) \rangle$, where
\begin{align*}
\langle x(x_0,0) \, ,\, u(x_0,0) \rangle&=\langle x_0 \, ,\, g(x_0) \rangle, \quad \text{and,}\\
\frac{\partial}{\partial t}\langle x(x_0,t) \, ,\, u(x_0,t) \rangle&=\langle F'(u), Q(u,x,t) \rangle.
\end{align*}
The curve $\langle x(x_0,t) \, ,\, u(x_0,t) \rangle$ remains a parametrization of the strong solution to (\ref{Cauchy}) up to time $t^*$, provided $\frac{\partial}{\partial x_0}x(x_0,t)>0$ for all $x_0$ in the computation domain and $0\leq t<t^*$. If at some point $x_0^*$ we have $\frac{\partial}{\partial x_0}x(x_0^*,t)<0$, for $t>0$, then the parametric curve becomes multi-valued and a projection is required to recover the appropriate weak solution to (\ref{Cauchy}). In the homogeneous setting it is common to employ an equal-area projection ( also known as the equal-area principle ), see \cite{LeVeque,LevFinite} and Figure \ref{Proj Sketch} for an illustration. In the non-homogeneous case however, it is unclear which projection yields the desired weak solution. In the present work, we construct a method able to capture the shock location, as given by a temporal integration of the Rankine-Hugoniot condition \cite{Rankine,Hugo}, to high-order, even in the non-homogeneous case. Most importantly, since the curves to the left and right of the shock are obtained through the parametric interpolation of (\ref{CharEq}) to high temporal and spatial order, we are therefore able to show that the weak solution is obtained to high-order as well.
\begin{figure}[!ht]
\begin{center}
\includegraphics[width=40mm,height=30mm]{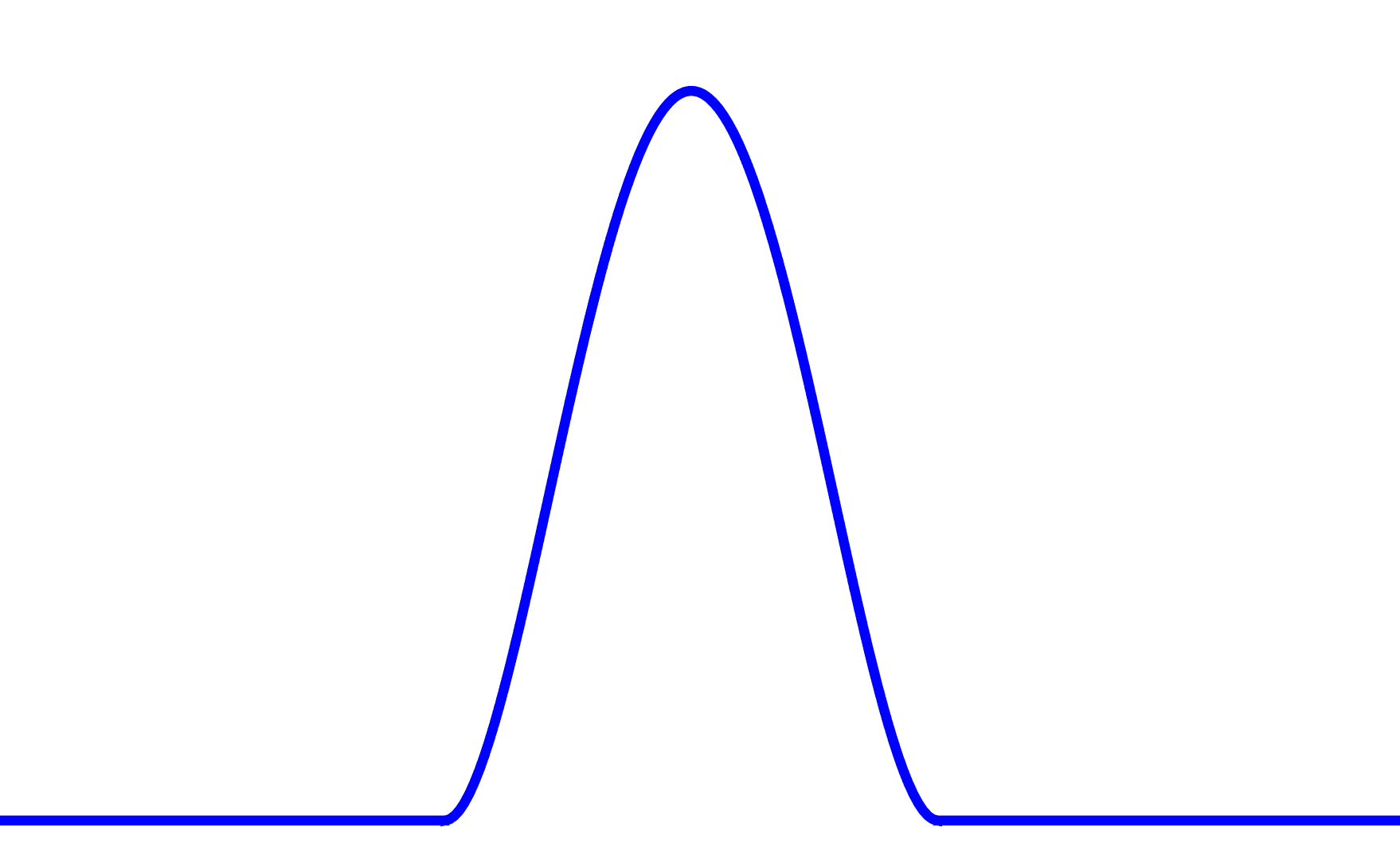}
\includegraphics[width=40mm,height=30mm]{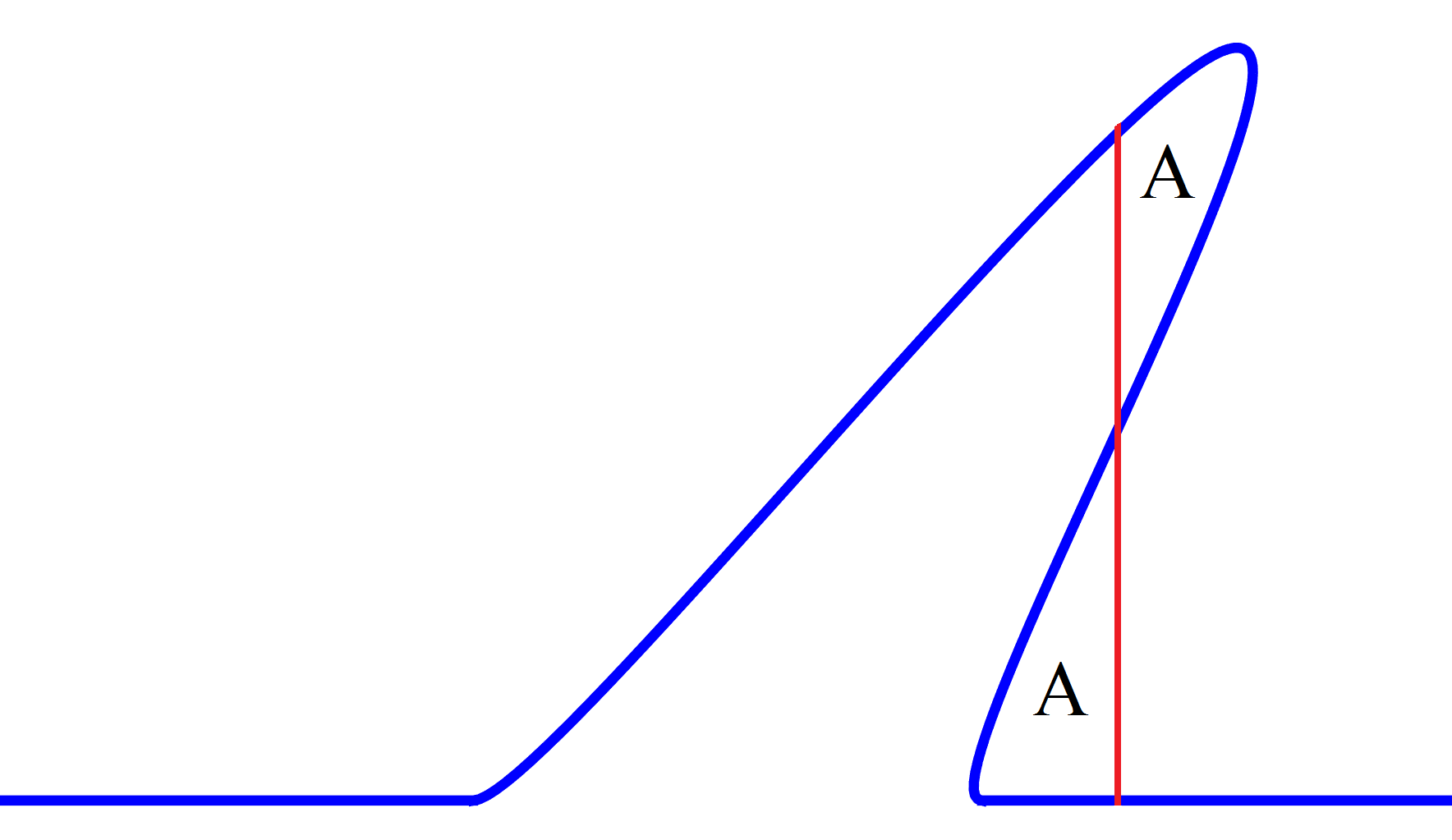}
\includegraphics[width=40mm,height=30mm]{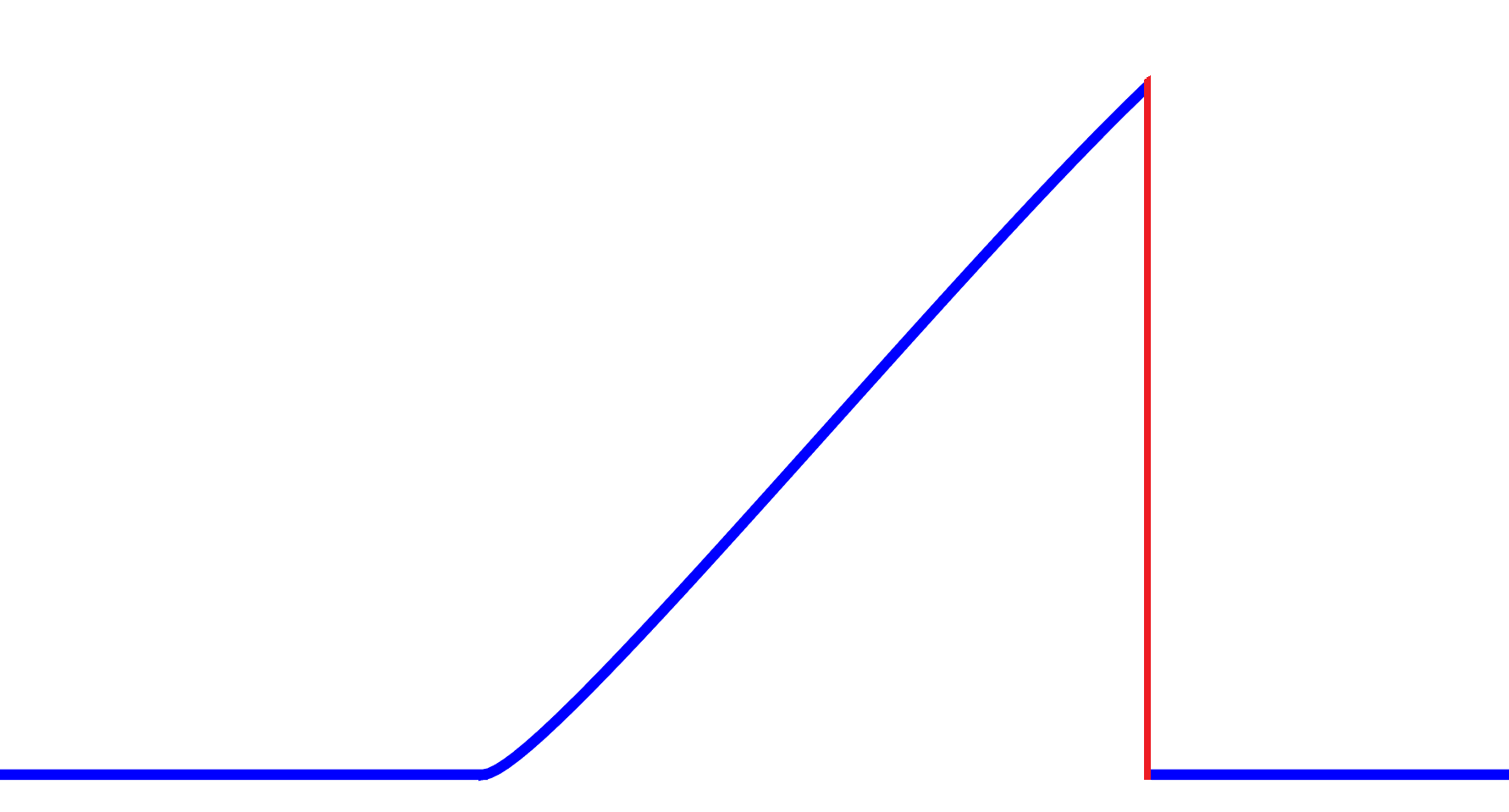}
\end{center}
\caption{ An illustration of the equal-area projection. }
\label{Proj Sketch}
\end{figure} 

At the core our approach is a parametric representation of the solution. The literature on parametric interpolation is extensive, however it is rarely utilized as a tool in numerical methods for differential equations. Given that we are concerned with high accuracy and smoothness of the solution curve, as opposed to the smoothness of the parametrization itself, we seek interpolation methods with high geometric continuity. Geometric continuity was first introduced in \cite{DeBoore}, where the authors matched function value, tangents and curvature at endpoints to obtain up to sixth order accuracy. Numerous other interpolation techniques can be employed to achieved desired characteristics, such as matching prescribed arc length, see \cite{ArcLength,farouki1}, or minimizing the curvature variation energy in \cite{CurvMin}, or the strain energy in \cite{StrainEnergy}.  Given that we are interested in obtaining high-order numerical schemes, we focus on parametric interpolation methods which emphasize accuracy.

In the homogeneous case of (\ref{Cauchy}), conservative methods are desirable, therefore we are interested in applying the area-preserving method developed in \cite{mcgregor2019area}. Here the authors construct a family of exactly area-preserving parametric Hermite polynomials which are fifth order accurate, one order higher than the standard parametric Hermite cubic. The conservative nature of these interpolants are particularly desirable in the homogeneous setting since the area change can only occur if there is flux over the boundary of the computation domain.

The paper is organized as follows: In Section \ref{PIF} we discuss the application of parametric interpolation to homogeneous 1-D scalar conservation laws. Specifically, Subsection \ref{AreaInterpolation} presents a brief overview of Bezier cubic interpolation and the area-preserving interpolation discussed in \cite{mcgregor2019area}. In Subsection \ref{1DScalar} we show how the methods from Section \ref{AreaInterpolation} can be applied to scalar conservation laws. Here the equal-area principle is discussed and we present a simple example showing that sixth order accuracy on shock position is obtained through the use of the interpolation framework in \cite{mcgregor2019area}. Sections \ref{NonHomCase} and \ref{Numerics} discuss the non-homogeneous problem, in particular these sections cover a modified equal-area principle and a shock propagation method which allows capturing shock position to high-order. Finally in Section \ref{Results}, full numerical results are presented with concluding remarks in Section \ref{Discussion}.

\section{Parametric interpolation framework for 1-D homogeneous scalar conservation laws}\label{PIF}

\subsection{Area-preserving parametric interpolation}\label{AreaInterpolation}
We begin by presenting a brief overview of the area-preserving B\'ezier interpolation discussed in \cite{mcgregor2019area}. The objective of this paper was as follows: given a planar parametric curve $\langle f(s), g(s) \rangle$, parametrized by $s\in[s_0,s_1]$, find a cubic B\'ezier polynomial defined by
\begin{equation}
\vec{B}(t)=\langle B_1(t) \, , \, B_2(t) \rangle=\vec{A}(1-t)^3+3\vec{C}_1(1-t)^2t+3\vec{C_2}(1-t)t^2+\vec{D}t^3, \quad \text{for $t\in[0,1]$},\label{Bezier}
\end{equation}
which satisfies
\begin{align*}
&\vec{B}(0)=\langle f(s_0) \, , \, g(s_0)\rangle, \quad \vec{B}(1)=\langle f(s_1) \, , \, g(s_1)\rangle,\nonumber\\
&\vec{B}'(0)=r_1\langle f'(s_0) \, , \, g'(s_0)\rangle=r_1\vec{\alpha},\quad \text{for some } r_2\in\mathbb{R}\nonumber\\
&\vec{B}'(1)=r_2\langle f'(s_1) \, , \, g'(s_1)\rangle=r_2\vec{\beta},\quad \text{for some } r_2\in\mathbb{R} , \text{and finally}\nonumber\\
&\int_0^1B_2(\tau)B_1'(\tau)\text{d}\tau=\int_{s_0}^{s_1}g(\tau)f'(\tau)\text{d}\tau.
\end{align*}
The coefficients $\vec{A},\vec{C}_1,\vec{C}_2$ and $\vec{D}$ are extracted from the functions $f$ and $g$ above, however, an additional degree of freedom remains.  After translating the data to the origin ($\vec{A}=0$), the integral condition above leads to a relation in terms of $r_1$ and $r_2$, given by
\begin{equation}
\frac{r_1r_2}{60}(\vec{\alpha}\times\vec{\beta} )+\frac{r_1}{10}(\vec{D}\times\vec{\alpha})+\frac{r_2}{10}(\vec{\beta}\times\vec{D})+\frac{D_1D_2}{2}=\int_{s_0}^{s_1}g(\tau)f'(\tau)\text{d}\tau-g(s_0)((f(s_1)-f(s_0)),\label{AreaPres}
\end{equation}
where $\times$ denotes the scalar vector product $\vec{\alpha}\times\vec{\beta}=\alpha_1\beta_2-\beta_1\alpha_2$. The main result in \cite{mcgregor2019area} proves that the interpolation is fifth order accurate in the $L^{\infty}$ norm provided the parameters $r_1$ and $r_2$ satisfy the area condition (\ref{AreaPres}) and an appropriate decay rate ( as $||\vec{D}-\vec{A}||\rightarrow 0 $), with the interpolation reducing to fourth order if the curvature vanishes somewhere within the domain of interpolation.  It is important to note that this result assumes the portion of the parametric curve being interpolated is small enough such that it may be represented by some function $\langle x,\tilde{f}(x)\rangle$ after an appropriate rotation.  The simplest fifth order area-preserving cubic B\'ezier interpolation of the function $\langle x,\tilde{f}(x)\rangle$ over the interval $x\in[0,h]$ is given by taking $r_1=h$ and solving for $r_2$ in (\ref{AreaPres}). This is a natural choice since the resulting curve can be viewed as a perturbed cubic Hermite polynomial. Given the data $\langle x,\tilde{f}(x)\rangle$ for $x\in[0,h]$ and taking the non-area preserving choice $r_1=r_2=h$ yields
\begin{align}
 \vec{B}_{H}(t)=\langle th, \tilde{f}(0)+\tilde{f}'(0)th&+\left(\frac{3(\tilde{f}(h)-\tilde{f}(0))-2\tilde{f}(0)h-\tilde{f}(h)h}{h^2}\right)(th)^2\nonumber\\
 &+\left(\frac{2(\tilde{f}(0)-\tilde{f}(h))+2\tilde{f}'(0)h+\tilde{f}(h)h}{h^3}\right)(th)^3 \rangle\label{BezierHermite}.
\end{align}
Defining $x=th$, we see that (\ref{BezierHermite}) is exactly a parametrization of the standard cubic Hermite polynomial. Therefore, as seen in Theorem 2.7 of \cite{mcgregor2019area}, when we take $r_1=h$ and let $r_2$ satisfy the area condition (\ref{AreaPres}) we obtain $r_2=h+\mathcal{O}(h^3)$, which implies the area-preserving cubic B\'ezier can be written as $\vec{B}(t)=\vec{B}_{H}(t)+\mathcal{O}(h^3)$, a perturbation from the cubic Hermite $\vec{B}_{H}(t)$ provided $h$ is small. We note that all of the fifth order area-preserving cubic B\'ezier's described in Theorem 2.7 of \cite{mcgregor2019area} can also be viewed as a perturbed cubic Hermite polynomial, the example given here is the simplest. If the provided curve $\langle x,\tilde{f}(x)\rangle$ for $x\in[0,h]$ is given in terms of a different parametrization with left and right tangents $\vec{\alpha}=\langle \alpha_1 , \alpha_2 \rangle$ and $\vec{\beta}=\langle \beta_1 , \beta_2 \rangle$ respectively, then taking $r_1=\frac{h}{\alpha_1}$ and $r_2=\frac{h}{\beta_1}$ yields the same cubic Hermite curve as (\ref{BezierHermite}). For more details on area-preserving cubic B\'ezier interpolation we refer the reader to \cite{mcgregor2019area}. Next we provide details on how to apply this parametric interpolation framework to homogeneous scalar conservation laws in one space dimension.

\subsection{Application to 1-D scalar conservation laws.}\label{1DScalar}

In this section we focus on the homogeneous case of (\ref{Cauchy}), specifically
\begin{equation}
\begin{cases}
u_t+(F(u))_x=0\label{CauchyHOM}\\
u(x,0)=g(x),
\end{cases}
\end{equation}
where $g$ is piecewise smooth and $F$ is both smooth and uniformly convex. We note that sufficient smoothness is only a requirement to obtain high-order convergence of the numerical methods discussed in this section.

Our first objective here is to justify that the parametric curve $\langle x(s,t), u(s,t) \rangle$ given by solving (\ref{CharEq}) can be used to obtain the correct weak solution. We also must show that we can obtain the required data from (\ref{CharEq}) to construct the area-preserving parametric polynomials described in \cite{mcgregor2019area}.

The Cauchy problem (\ref{CauchyHOM}) yields the characteristic equations
\begin{align}
\dot{x}&=F'(u) \label{CharEqHom}\\
\dot{u}&=0, \nonumber
\end{align}
which can be solved exactly, with
\begin{align}
x(x_0,t)&=x_0+F'(g(x_0))t\nonumber\\
u(x_0,t)&=g(x_0),\label{HomParCurve}
\end{align}
where $g(x)$ is the given initial condition of (\ref{CauchyHOM}). Written as a planar curve parametrized by $x_0$, we have the solution to (\ref{CharEqHom}) is given by $\langle x_0+F'(g(x_0))t \, , \, g(x_0) \rangle$. The method of characteristics, as discussed in \cite{Evans,LAX}, says that until the formation of a shock at time $t=t^*$ the strong solution of (\ref{CauchyHOM}) is given by the curve $\langle x_0+F'(g(x_0))t, g(x_0) \rangle $. Additionally, even once a shock has formed, the portions of the weak solution which on either side of the shock are also given (\ref{HomParCurve}). Therefore, if we are able to obtain the correct location of discontinuities in weak solutions of (\ref{CauchyHOM}), and everywhere else in the solution is given by the curve (\ref{HomParCurve}), then we must have the correct weak solution. We begin our discussion of weak solutions by addressing shock position and the equal-area principle.

The equal-area principle for scalar 1-D scalar conservation laws with convex flux functions has been utilized to obtain the location of shocks by many, see \cite{LevFinite, Equal}. However, a proof of its validity for all times $t\geq t^*$ is generally neglected. Therefore for completeness we include a proof that the equal-area principle is equivalent to the Rankine-Hugoniot condition in this setting, but first we must first introduce what we mean by an equal-area curve.

\begin{definition}
We say $\langle x(s), u(s) \rangle$ is an equal-area curve from $s_0$ to $s_1$ provided $x(s_0)=x(s_1)$ and
\begin{equation*}
\int_{s_0}^{s_1}u(s)x'(s)\text{d}s=0.
\end{equation*}
\end{definition}
Using this definition we are able to present the following Theorem on the equal-area principle.

\begin{theorem}\label{Thm1}
Suppose $\langle x(s,t),u(s) \rangle$ is a parametrization of (\ref{HomParCurve}), then \\ $\langle x(s,t),u(s) \rangle$ is an equal-area curve between $s_1(t)$ and $s_2(t)$ for $t\in [t^*,T]$ if and only if $\frac{d}{dt}x(s_1(t),t)=\frac{d}{dt}x(s_2(t),t)=\frac{F(u(s_1(t)))-F(u(s_2(t)))}{u(s_1(t))-u(s_2(t))}$ for $t\in [t^*,T]$.
\end{theorem}

\begin{proof}
Suppose $\langle x(s,t),u(s) \rangle=\langle s + F'(g(s))t , g(s) \rangle$ solves the system of characteristic equations (\ref{CharEqHom}) and is an equal-area curve between $s_1(t)$ and $s_2(t)$ on $t\in[t^*,T]$. This implies
\begin{equation}
\frac{d}{dt}\int_{s_1(t)}^{s_2(t)}u(s)x_s(s,t)\text{d}s=0,
\end{equation}
on $t\in[t^*,T]$. Expanding the above integral leads to
\begin{equation*}
u(s_2(t))x_s(s_2(t),t)s_2'(t)-u(s_1(t))x_s(s_1(t),t)s_1'(t)+\int_{s_1(t)}^{s_2(t)}u(s)x_{st}(s,t)\text{d}s=0.
\end{equation*}
The chain rule yields $\frac{d}{dt}x(s(t)(t),t)=x_s(s(t),t)s'(t)+x_t(s(t),t)$ and using equal-area curve property that $\frac{d}{dt}x(s_1(t),t)=\frac{d}{dt}x(s_2(t),t)$, the above equation can be rewritten as
\begin{align}
\frac{d}{dt}x(s_1(t),t)\left(u(s_2(t))-u(s_1(t))\right)+u(s_1(t))x_t(s_1(t),t)-u(s_2(t))x_t(s_2(t),t)&\nonumber\\
+\int_{s_1(t)}^{s_2(t)}u(s)x_{st}(s,t)\text{d}s=0&\label{EADer}.
\end{align}

Using $\langle x(s,t),u(s) \rangle=\langle s + F'(g(s))t , g(s) \rangle$ allows us to compute the integral term contained in (\ref{EADer}). A substitution and integration by parts yields
\begin{align}
\int_{s_1(t)}^{s_2(t)}u(s)x_{st}(s,t)\text{d}s&=\int_{s_1(t)}^{s_2(t)}g(s)F''(g(s))g'(s)\text{d}s\nonumber\\
&=\left(u(s)F'(u(s))-F(u(s))\right)\biggr|_{s=s_1(t)}^{s_2(t)}.\label{FluxInt}
\end{align}
Subbing this into (\ref{EADer}) and replacing $x_t(s(t),t)=F'(u(s(t))$ yields
\begin{align*}
\frac{d}{dt}x(s_1(t),t)\left(u(s_2(t))-u(s_1(t))\right)+F(u(s_1(t)))-F(u(s_2(t)))=0,
\end{align*}
which implies that the position of the equal-area curve moves at the correct Rankine-Hugoniot speed.

To prove the converse we start by assuming we have an isolated shock formed at time $t=t^*$ and position $x(s_1(t^*),t^*)=x(s_2(t^*),t^*)$  and $\frac{d}{dt}x(s_1(t),t)=\frac{d}{dt}x(s_2(t),t)=\frac{F(u(s_1(t)))-F(u(s_2(t)))}{u(s_1(t))-u(s_2(t))}$ for $t\in [t^*,T]$. Repeating the above calculations we obtain that 
\begin{equation}
\frac{d}{dt}\int_{s_1(t)}^{s_2(t)}u(s)x_s(s,t)\text{d}s=0.
\end{equation} 
 To show that this is an equal-area curve we must prove that $\displaystyle \int_{s_1(t)}^{s_2(t)}u(s)x_s(s,t)\text{d}s=0$. At $t=t^*$ there are two possibilities, first we could have that $s_1(t^*)=s_2(t^*)$, which would prove the desired result. The other possibility is that $s_1(t^*)<s_2(t^*)$. If $1+F''(g(s))g'(s)t^*=0$ everywhere in the interval $[s_1(t^*),s_2(t^*)]$ then we also have an equal-area curve. The only possibility for us to not have an equal-area curve is if $1+F''(g(s))g'(s)t\neq0\Rightarrow x_s(s,t^*)\neq0$ somewhere in $s\in[s_1(t^*),s_2(t^*)]$. Since $x(s_1(t^*))=x(s_2(t^*))$ this would imply $x_s(s,t^*)>0$ somewhere in the interval and $x_s(s,t^*)<0$ somewhere in the interval, which implies that a shock must have emerged at an earlier time, which contradicts the definition of $t^*$. This completes the proof.
\end{proof}


We have now shown that the parametric curve given by (\ref{HomParCurve}) can be used to construct the correct weak solution of (\ref{CauchyHOM}) provided an appropriate projection is used. Next we show how to extract the data from (\ref{HomParCurve}) to construct the area-preserving parametric interpolation of \cite{mcgregor2019area}.

Consider the parametric curve parametrized by $x_0$ given by $\langle x(x_0,t) , u(x_0) \rangle=\langle x_0+F'(g(x_0))t , g(x_0) \rangle$.  Recall that utilizing the fifth order interpolation method of \cite{mcgregor2019area} requires endpoint values, tangents directions and the parametric area for each interpolant. As we will be performing a piecewise interpolation of this curve we partition computation domain for $x_0$ into $n$ subintervals $[x_i,x_{i+1}]$, for $i=1,..n$. In the $i^{th}$ subinterval at time $t=\tau$ we construct the $i^{th}$ interpolant $\langle x_i(s)  ,  y_i(s)\rangle$ satisfying
\begin{align*}
&\langle x_i(0) , y_i(0)\rangle=\langle x_i+F'(g(x_i))\tau,g(x_i)\rangle,\\
&\langle x_i(1)  ,  y_i(1)\rangle=\langle x_{i+1}+F'(g(x_{i+1}))\tau  ,  g(x_{i+1})\rangle,\\
&\langle x_i'(0) ,  y_i'(0)\rangle=r_{1_i}\langle 1+F''(g(x_i))g'(x_i)\tau  ,  g'(x_i)\rangle=r_{1_i}\vec{\alpha}_i,\quad \text{for some } r_{1_i}\in\mathbb{R}\\
&\langle x_i'(1) ,  y_i'(1)\rangle=r_{2_i}\langle 1+F''(g(x_{i+1}))g'(x_{i+1})\tau  ,  g'(x_{i+1})\rangle=r_{2_i}\vec{\beta}_i,\quad \text{for some } r_{2_i}\in\mathbb{R},\\
&\int_0^1y(s)x'(s)\text{d}s=\int_{x_i}^{x_{i+1}}g(s)(1+F''(g(s))g'(s)\tau)\text{d}s.
\end{align*}
From the above equations it is clear that accessing the endpoint values and tangents simply requires the evaluation of known functions given in (\ref{CauchyHOM}). We note that the final integral condition can be greatly simplified through a substitution and integration by parts to obtain
\begin{align*}
\int_0^1y(s)x'(s)\text{d}s&=\int_{x_i}^{x_{i+1}}g(s)\text{d}s\\
&+\tau\left(F'(g(x_{i+1}))g(x_{i+1})-F'(g(x_{i}))g(x_{i})+F(g(x_i))-F(g(x_{i+1}))\right).
\end{align*}
Therefore updating the integral data from time $t$ to $t+\Delta t$ simply requires an update of the second term, implying that integration is not required after the initial step. 

\subsection{Example 1}\label{Ex1}

In this example we consider the following Cauchy problem
\begin{equation}
\begin{cases}
u_t+uu_x=0\label{CauchyEx1}\\
u(x,0)=g(x),
\end{cases}
\end{equation}
where,
\begin{equation*}g(x)=
\begin{cases}
\sin(x), \quad \text{for $x\in[0,\pi]$}\\
0 \quad \text{otherwise}.
\end{cases}
\end{equation*} 

We begin by constructing the area-preserving interpolation of the parametric curve obtained from the characteristic equation and show that we obtain the convergence claimed in \cite{mcgregor2019area}. Then, we apply the equal-area principle discussed in Theorem \ref{Thm1} and show that the correct shock position is obtained. 

The characteristic equations associated with (\ref{CauchyEx1}) are
\begin{align}
\dot{x}&=u \label{CharEqEx1}\\
\dot{u}&=0, \nonumber
\end{align}
which yields the solution
\begin{equation}
\langle x(x_0,t),u(x_0) \rangle = \langle x_0+\sin(x_0)t, \sin(x_0) \rangle, \quad \text{for $x\in[0,\pi]$}.\label{Ex1Sln}
\end{equation}

For this simple example we are able to compute the equal-area solution by hand. A plot of the area-preserving parametric interpolation of (\ref{Ex1Sln}) is shown below in Figure \ref{SineEx} at times $t=0$, $t=1$ and $t=2$. The resulting weak solution is obtained by cutting out the overturn portion of the curve at the shock location.

\begin{figure}[!ht]
\begin{center}
\includegraphics[width=40mm,height=30mm]{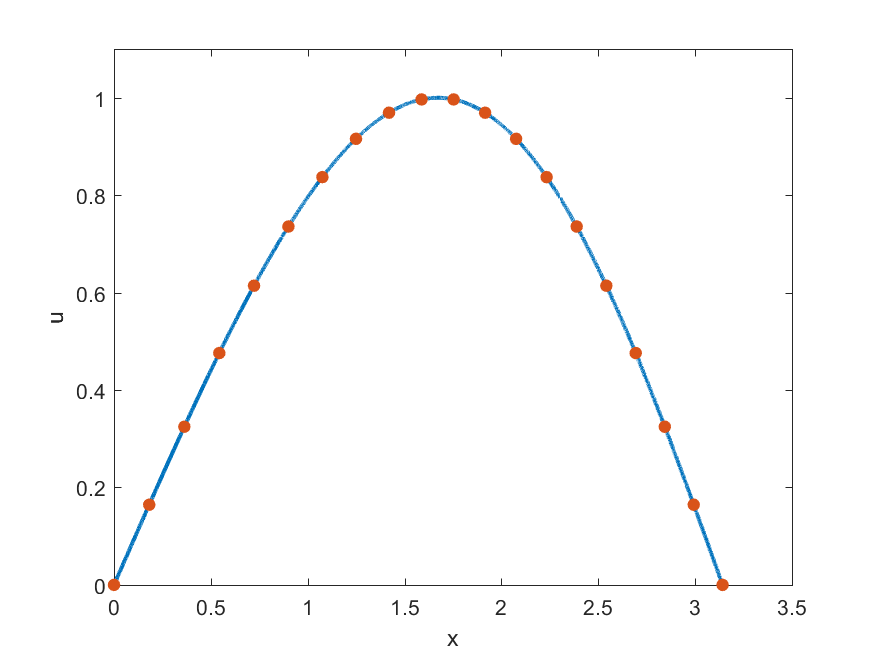}
\includegraphics[width=40mm,height=30mm]{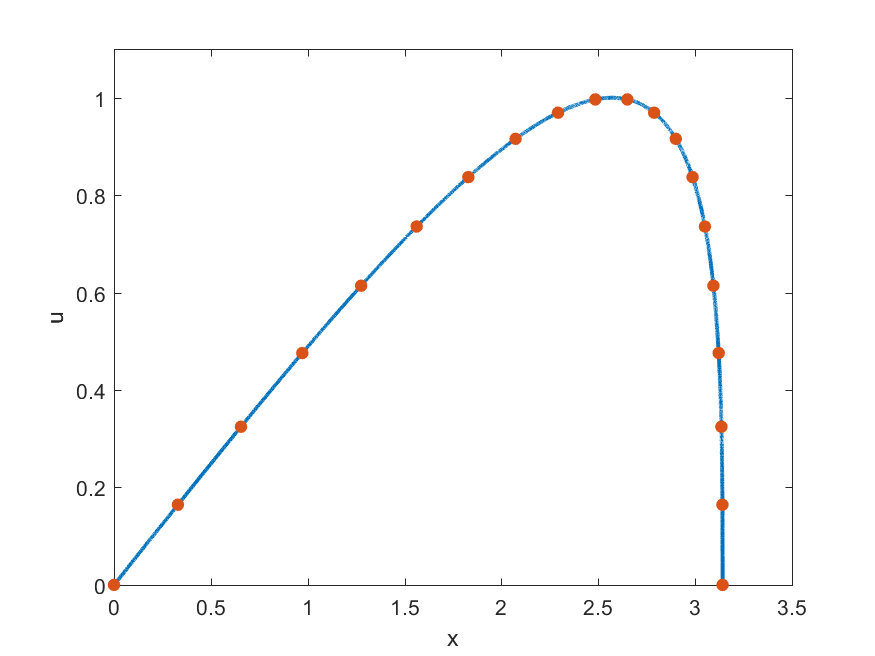}
\includegraphics[width=40mm,height=30mm]{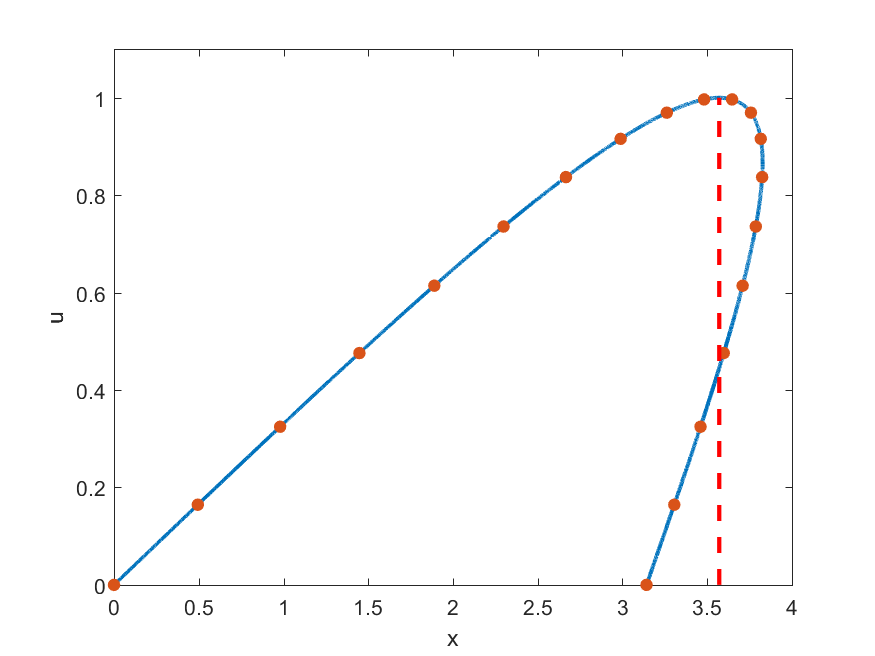}
\end{center}
\label{SineEx}
\caption{ Area-preserving interpolation of (\ref{Ex1Sln}) }
\end{figure} 

As we can see the curve eventually overturns in such a way that we know the right state of the shock will be height zero. With this information we search for the point in the parametrization $a(t)$ satisfying
\begin{equation*}
\int_{a(t)}^{\pi}\sin(s)(1+\cos(s)t)\text{d}s=0,
\end{equation*}
which, after some simplification, leads to the equation
\begin{equation*}
\frac{t\cos^2(a(t))}{2}+\cos(a(t))+\left(1-\frac{t}{2}\right)=0.
\end{equation*}
Some basic algebra tells us that $a(t)=\arccos\left(\frac{t-2}{t}\right)$, and that the position of the equal-area line is therefore given by $\displaystyle x(a(t),t)=\arccos\left(\frac{t-2}{t}\right)+2\sqrt{t-1}$. The height of the equal-area line is given by $\displaystyle u(a(t))=\sin\left(\arccos\left(\frac{t-2}{t}\right)\right)=2\frac{\sqrt{t-1}}{t}$. To see that this agrees with the Rankine-Hugoniot condition we should have that $\displaystyle \frac{d}{dt}x(a(t),t)=\frac{F(u(a(t))}{u(a(t))}=\frac{u(a(t))}{2}=\frac{\sqrt{t-1}}{t}$. Differentiating $\arccos\left(\frac{t-2}{t}\right)+2\sqrt{t-1}$ with respect to $t$ shows this is indeed the case. Therefore the equal-area line moves at Rankine-Hugoniot speed. Next we show how one constructs the area-preserving parametric polynomials for this example.

On each subinterval $[x_{i},x_{i+1}]\subset[0,\pi]$, for $i=1, \dots, n-1$ we use (\ref{Ex1Sln}) to construct our B\'ezier polynomial $\vec{B}_i(t)$,  (\ref{Bezier}). The coefficients $\vec{A}_i,\vec{\alpha}_i,\vec{\beta}_i$ and $\vec{D}_i$ are given by 
\begin{align*}
\vec{A}_i &= \langle x_i+\sin(x_i)t, \sin(x_i) \rangle\\
\vec{D}_i &= \langle x_{i+1}+\sin(x_{i+1})t, \sin(x_{i+1}) \rangle\\
\vec{\alpha}_i &= \langle 1+\cos(x_{i})t, \cos(x_i) \rangle\\
\vec{\beta}_i &= \langle 1+\cos(x_{i+1})t, \cos(x_{i+1}) \rangle,
\end{align*}
where $\vec{C_1}_i=\vec{A}_i+\frac{r_{1_i}\vec{\alpha}_i}{3}$, and $\vec{C_2}_i=\vec{D}_i-\frac{r_{2_i}\vec{\beta}_i}{3}$. Therefore all that is left is to determine $r_{1_i}$ and $r_{2_i}$. As discussed in \cite{mcgregor2019area} there is an entire family of pairs $(r_{1_i},r_{2_i})$ which are fifth order accurate and ensure exact area-preservation. As discussed in Section \ref{PIF} we take $r_{1_i}=\frac{D_{1_i}-A_{1_i}}{\alpha_1}$ and solve for $r_{2_i}$ using (\ref{AreaPres}).

The error in shock position at time $t=2$ is plotted in Figure \ref{SineExampleConv}. We expect to obtain sixth-order accuracy in the shock position since it is obtained from the integration of a fifth order accurate curve. As predicted, Figure \ref{SineExampleConv} shows that sixth order accuracy in the shock position is obtained.

\begin{figure}[!ht]
\begin{center}
\includegraphics[width=80mm,height=60mm]{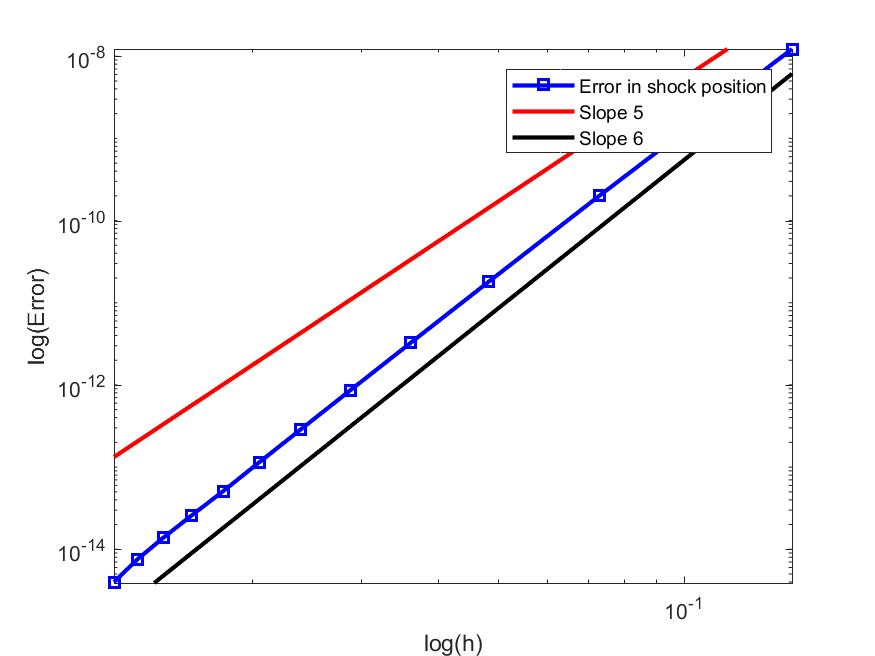}
\end{center}
\caption{ Area-preserving interpolation of (\ref{Ex1Sln}) }
\label{SineExampleConv}
\end{figure}

We see that in the homogeneous case the data for constructing the parametric interpolants is readily available, making our numerical scheme simple to implement. Moving to the source term case we lose a bit of this efficiency as, in general, we cannot obtain an analytical solution to the characteristic equations (\ref{CharEq}). On top of this the shock motion itself is far more complex, however, through the use of our parametric interpolation framework we are still able to capture the shock position to high-order, both temporally and spatially.

\section{The non-homogenous case}\label{NonHomCase}
In this section we apply the parametric framework discussed in Section \ref{AreaInterpolation} to the non-homogeneous setting. The Cauchy problem we are interested in solving is given by
\begin{equation}
\begin{cases}
u_t+(F(u))_x=Q(u,x,t)\label{Cauchy2}\\
u(x,0)=g(x),
\end{cases}
\end{equation}
where $g$ is piecewise smooth and both $F$ and $Q$ are smooth functions in their respective domains with $F$ uniformly convex. As discussed in Section \ref{Intro}, we apply the method of characteristics to obtain the system of equations
\begin{align}
\dot{x}&=F'(u) \label{CharEq2}\\
\dot{u}&=Q(u,x,t). \nonumber
\end{align}
The solution to (\ref{CharEq2}) can be represented as the parametric curve parametrized by $x_0$, $\langle x(x_0,t) , u(x_0,t) \rangle$, or in terms of $F$ and $Q$ as
\begin{equation}
\langle x(x_0,t) , u(x_0,t) \rangle=\langle x_0+\int_0^t F'(u)\text{d}\tau \, , \, g(x_0)+\int_0^t Q(u,x,\tau)\text{d}\tau \rangle.\label{ParNonHom}
\end{equation}

The distinction from the homogeneous case appears through the second component, $u$, which now varies in time. Also we notice that the system (\ref{CharEq2}) is a fully coupled system of ordinary differential equations, and therefore we are not able to come up with a general solution without knowing more about $F$ and $Q$. This implies that, in general, we will work directly with the curve (\ref{ParNonHom}) to extract the required data for the parametric interpolation, and thus introduce a temporal error that was not present in the homogeneous case. For now we focus our attention on the validity of the parametric framework before we discuss how to construct the parametric interpolants in this setting.

Just as in the homogeneous case, if the solution curve $\langle x(x_0,t) , u(x_0,t) \rangle$ obtained by solving (\ref{CharEq2}) does not overturn (remaining the graph of a single variable function) then the method of characteristics guarantees that this is indeed the correct solution to our Cauchy problem (\ref{Cauchy2}). Therefore we only need to worry about the case when discontinuities are present in the solution. In the homogeneous case we found that the equal-area principle provided us with the correct projection to obtain the desired weak solution of (\ref{CauchyHOM}), however, as seen in the next theorem this same approach does not work for general source terms $Q(u,x,t)$.

\begin{theorem}\label{NonHomEqualArea}
The equal-area principle applied to the parametric curve (\ref{ParNonHom}), in general, fails to predict the correct shock speed and thus cannot be used to find the appropriate weak solution of (\ref{Cauchy2}).
\end{theorem}
\begin{proof}
Applying the same technique as in Theorem \ref{Thm1} we begin by supposing we have an equal-area curve between $s_1(t)$and $s_2(t)$. This implies that
\begin{equation*}
\frac{d}{dt}\int_{s_1(t)}^{s_2(t)}u(s,t)x_s(s,t)\text{d}s=0.
\end{equation*}
Computing the full derivative we obtain a very similar result as in the homogeneous case, with 
\begin{align*}
u(s_2(t),t)x_s(s_2(t),t)s_2'(t)-u(s_1(t),t)x_s(s_1(t),t)s_1'(t)&\\
+\int_{s_1(t)}^{s_2(t)}u(s,t)x_{st}(s,t)\text{d}s+\int_{s_1(t)}^{s_2(t)}Q(u,x,t)x_{s}(s,t)\text{d}s&=0.
\end{align*}
The major difference here is the addition of the second integral term which, in general, we are unable to simplify. Continuing with the same simplifications as in Theorem \ref{Thm1}, we set $\frac{d}{dt}x(s(t),t)-x_t(s(t),t)=x_s(s(t),t)s'(t)$.  Recalling that $x(s_1(t),t)=x(s_2(t),t)$ and using $\langle x(x_0,t),u(x_0,t)\rangle$ from (\ref{ParNonHom}) we obtain
\begin{align}
\frac{d}{dt}x(s_1(t),t)\left(u(s_2(t),t)-u(s_1(t),t)\right)+u(s_1(t),t)F'(u(s_1(t),t))&\\-u(s_2(t),t)F'(u(s_2(t),t))
+\int_{s_1(t)}^{s_2(t)}u(s,t)F''(u(s,t))u_s(s,t)\text{d}s&\nonumber\\+\int_{s_1(t)}^{s_2(t)}Q(u,x,t)x_{s}(s,t)\text{d}s=0&\nonumber\label{EADer2}.
\end{align}
The first of the two integrals is the same as (\ref{FluxInt}), therefore applying the same procedure as in Theorem \ref{Thm1} we obtain the final equation
\begin{align}
\frac{d}{dt}x(s_1(t),t)=\frac{F(u(s_1(t),t))-F(u(s_2(t),t))}{u(s_1(t),t)-u(s_2(t),t)}+\frac{\int_{s_1(t)}^{s_2(t)}Q(u,x,t)x_{s}(s,t)\text{d}s}{u(s_1(t),t)-u(s_2(t),t)}.\label{NonHomEAP}
\end{align}
Therefore, if at any point in time the integral term in (\ref{NonHomEAP}) is nonzero we obtain the incorrect shock speed from the equal-area principle, which completes the proof.
\end{proof}

To better understand why the equal-area principle fails in the non-homogeneous case we look to the following example.

\subsection{Example 2}\label{Ex2}
Here we consider the Cauchy problem
\begin{equation}
\begin{cases}
u_t+uu_x=-(u(1-u))^k\label{CauchyEx2}\\
u(x,0)=g(x),
\end{cases}
\end{equation}
where,
\begin{equation*}g(x)=
\begin{cases}
1, \quad \text{for $x\in[0,1]$}\\
0 \quad \text{otherwise},
\end{cases}
\end{equation*}
where $k$ is a parameter of our choosing.

Mirroring the technique used in the homogeneous setting, we take each point on the initial condition sketched in Figure \ref{BoxExIC} (both the top and sides of the rectangle) and flow them under the characteristic equations associated with (\ref{CauchyEx2}). In terms of the free parameter $k$, the characteristic equations are
\begin{align}
\dot{x}&=u \label{CharEqEx2}\\
\dot{u}&=-(u(1-u))^k. \nonumber
\end{align}
\begin{figure}[!ht]
\begin{center}
\includegraphics[width=50mm,height=30mm]{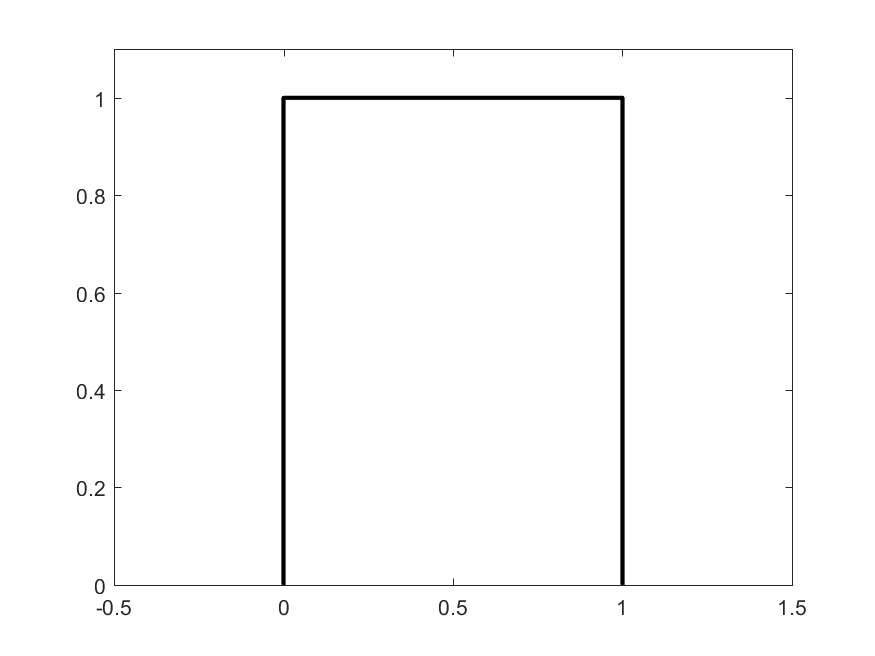}
\end{center}
\caption{ Initial condition for Example \ref{Ex2}. }
\label{BoxExIC}
\end{figure}
Under the dynamics governed by the system (\ref{CharEqEx2}) we flow each portion of the initial condition in Figure \ref{BoxExIC}  until $t=1$ for the parameter values $k=1, 1.5$ and $6$. The resulting curves are displayed in Figure \ref{BoxEx}. The rarefaction curve on the left portion of Figure \ref{BoxEx} is correct for the respective values of $k$, and the portion of the curve on top is also correct, however a shock should be present somewhere in the multi-valued portion of the curve. If we apply the equal-area principle it is clear, even by visual inspection, that each of these three curves will yield a different equal-area line and thus the equal-area principle will predict different shock positions at time $t=1$ for the different values of $k$. However, $u=1$ is a fixed point of its corresponding differential equation regardless of the value of $k$, therefore we know that the shock speed will be exactly $\frac{1}{2}$ until the rarefaction wave comes into contact with the shock, but this doesn't occur until $t=2$. Therefore the equal-area principle clearly is not capturing the motion of the shock in the non-homogeneous case correctly.

\begin{figure}[!ht]
\begin{center}
\includegraphics[width=100mm,height=60mm]{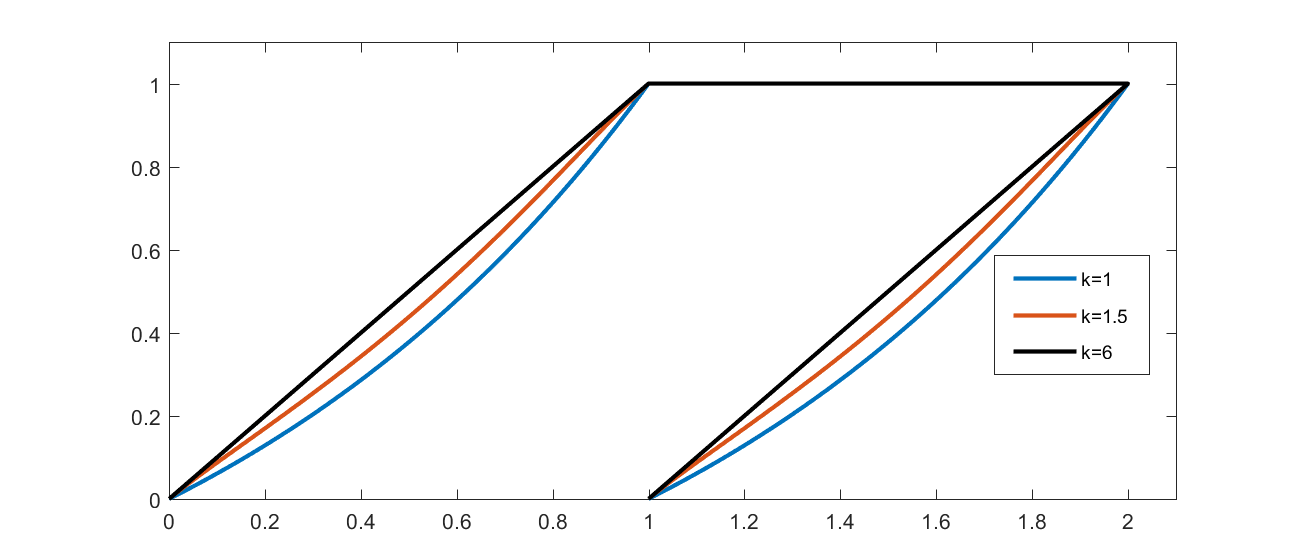}
\end{center}
\caption{Resulting curves after solving the system (\ref{CharEqEx2}) for $k=1, 1.5$ and $6$ until $t=1$.}
\label{BoxEx}
\end{figure} 

The main takeaway from Theorem \ref{NonHomEqualArea} and Example \ref{Ex2} is that shock formation and its subsequent motion need to be treated carefully and cannot be computed with the same techniques as in the homogeneous case. In the following section we present numerical methods for first detecting and locating where shocks are forming and then how to compute their motion to high spatial and temporal order. 

\section{Numerical methods for shock motion in the non-homogeneous case}\label{Numerics}

The numerical framework presented in this section will contain three distinct parts. First we have the standard characteristic flow given by solving (\ref{CharEq2}) on which we perform a high-order parametric interpolation. The second part of the method involves shock detection and initial shock positioning which will utilize an equal-area projection along with a splitting method. Finally we have the shock propagation method. For these methods to be effective we rely heavily on having a precise representation of the solution curve associated with (\ref{CharEq2}). We therefore begin this section by briefly discussing how to obtain high-order parametric interpolants in the non-homogeneous case.

Recalling again the characteristic equations and the given initial condition,

\begin{equation}\label{CharEq3}
 \left\{
\begin{array}{l}
     \dot{x}=F'(u)\\
     \dot{u}=Q(u,x,t)\\
x(0)=x_0\\
u(0)=g(x_0), 
\end{array}\right.
\end{equation}

we aim to construct solutions to (\ref{CharEq3}) at each point $x_0$ in the domain of $g(x_0)$ given in the Cauchy problem (\ref{Cauchy}). As mentioned in Section \ref{NonHomCase} we are unable to solve these equations exactly in the vast majority of cases, therefore we apply numerical methods on the equations (\ref{ParNonHom}),
\begin{equation*}
\langle x(x_0,t) , u(x_0,t) \rangle=\langle x_0+\int_0^t F'(u)\text{d}\tau \, , \, g(x_0)+\int_0^t Q(u,x,\tau)\text{d}\tau \rangle.
\end{equation*}
At each grid point $x_i$ we employ a suitable differential equation solver to obtain $\langle \tilde{x}(x_i,t) , \tilde{u}(x_i,t) \rangle$, for example, if the differential equations are not stiff, explicit Runge-Kutta methods can be utilized. Applying this idea on consecutive nodes $x_i$ and $x_{i+1}$ yields the endpoint values for us to perform our interpolation. To obtain our parametric interpolation to at least fourth order we also need information about the tangents along the solution curve. This data is obtained by expanding the system (\ref{CharEq3}) to include differential equations for the partials along the parametrization. Our extended set of characteristic equations therefore becomes
\begin{equation}\label{CharEqExt}
 \left\{
\begin{array}{l}
     \dot{x}=F'(u)\\
     \dot{x}_{x_0}=F''(u)u_{x_0}\\
     \dot{u}=Q(u,x,t)\\
\dot{u}_{x_0}=Q_u(u,x,t)u_{x_0}+Q_x(u,x,t)x_{x_0}\\
x(0)=x_0\\
x_{x_0}(0)=1 \\
u(0)=g(x_0) \\
u_{x_0}(0)=g'(x_0).
\end{array}\right.
\end{equation}
Solving (\ref{CharEqExt}) at initial points $x_i$ and $x_{i+1}$ until time $t$ yields the endpoint data $\tilde{x}(x_i,t), \tilde{u}(x_i,t), \tilde{x}(x_{i+1},t), \tilde{x}(x_{i+1},t)$ along with tangent the data $\tilde{x}_{x_0}(x_i,t)$, $\tilde{u}_{x_0}(x_i,t)$, $\tilde{x}_{x_0}(x_{i+1},t)$ and $\tilde{x}_{x_0}(x_{i+1},t)$. Using the cubic B\'ezier framework discussed in Section \ref{PIF} we only need to determine choices for $r_{1_i}$ and $r_{2_i}$ to generate the parametric polynomial interpolants. A fourth order accurate choice is simply using the parametric Hermite from Section \ref{AreaInterpolation},
\begin{align}
r_{1_i}&=\frac{\tilde{x}(x_{i+1},t)-\tilde{x}(x_{i},t)}{\tilde{x}_{x_0}(x_{i},t)}, \quad \text{provided $\tilde{x}_{x_0}(x_{i},t)\neq 0$}, \label{r1Herm}\\
r_{2_i}&=\frac{\tilde{x}(x_{i+1},t)-\tilde{x}(x_{i},t)}{\tilde{x}_{x_0}(x_{i+1},t)}, \quad \text{provided $\tilde{x}_{x_0}(x_{i+1},t)\neq 0$}. \label{r2Herm}
\end{align}
If either of the horizontal tangents vanish, then a simple rotation of the problem or a refinement of the grid is required.

\begin{example}\label{Ex3}
In this example we want to perform a high-order interpolation ( both in space and time ) of the particle path traced out by the solution to the initial value problem
\begin{equation}\label{Ex3IC}
 \left\{
\begin{array}{l}
     \dot{x}=u\\
     \dot{u}=\sin(x)u\\
     x(0)=0\\
     u(0)=\frac{1}{2}.
\end{array}\right.
\end{equation}
The chain rule tells us that $u(x(t))=\frac{3}{2}-\cos(x(t))$, which implies $\frac{d}{dt}x=\frac{3}{2}-\cos(x(t))$. With a bit of work one arrives at the solution
\begin{align}
x(t)&=2\arctan\left(\frac{\tan\left(\frac{\sqrt{5}}{4}t\right)}{\sqrt{5}}\right)\label{Ex3Sln}\\
u(t)&=\frac{3}{2}-\cos\left(2\arctan\left(\frac{\tan\left(\frac{\sqrt{5}}{4}t\right)}{\sqrt{5}}\right)\right)\nonumber.
\end{align}
The curve from (\ref{Ex3Sln}) at times $t=2.5$, $t=3$ and $t=5$ is plotted below in Figure \ref{FigODESolve}.

\begin{figure}[!ht]
\begin{center}
\includegraphics[width=40mm,height=40mm]{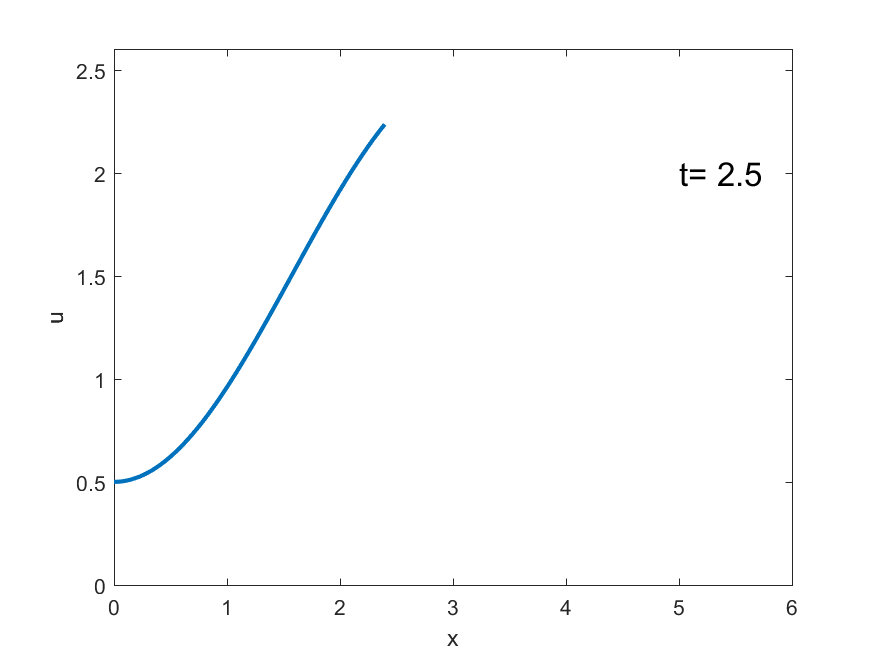}
\includegraphics[width=40mm,height=40mm]{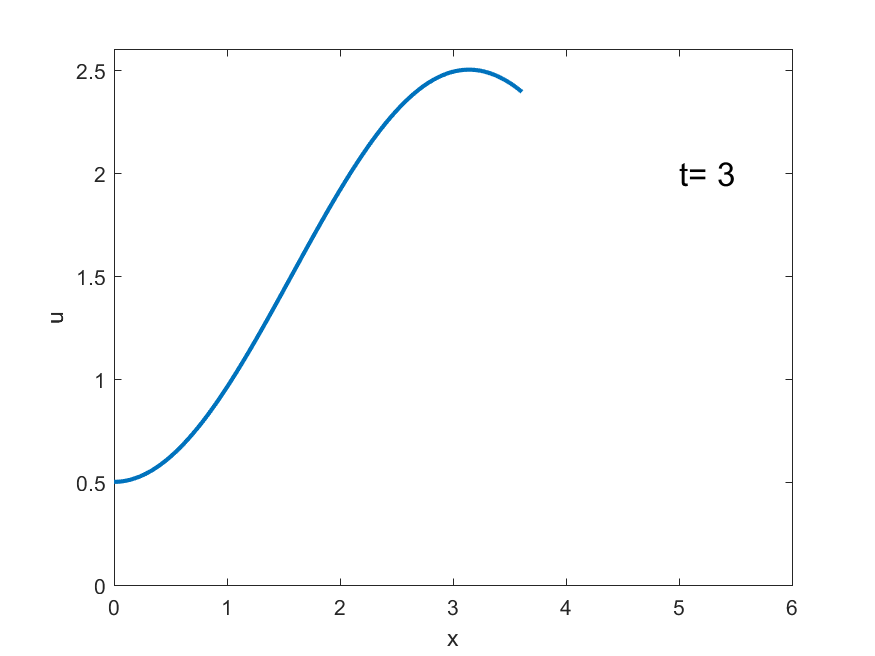}
\includegraphics[width=40mm,height=40mm]{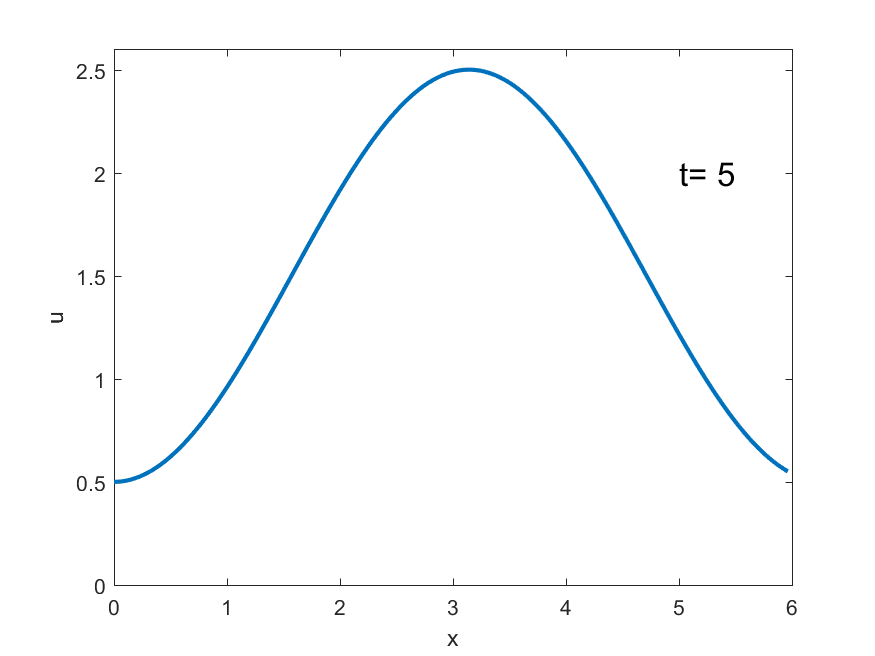}
\end{center}
\caption{ Particle path of the solution to the initial value problem (\ref{Ex3IC}) at times $t=2.5$, $t=3$ and $t=5$ respectively.}
\label{FigODESolve}
\end{figure} 
Using Runge-Kutta $4$ and a time step of $\Delta t =0.001$ we compute the particle position at time $t=5$ to fourteen digits of accuracy.  The data required to construct the parametric interpolants is given through the Runge-Kutta method applied to the extended system (\ref{CharEqExt}) associated with (\ref{Ex3IC}).  Constructing the parametric cubic Hermite polynomials discussed above we achieve fourth order convergence in space as seen in Figure \ref{HermSpatial}.
 
 \begin{figure}[!ht]
\begin{center}
\includegraphics[width=60mm,height=40mm]{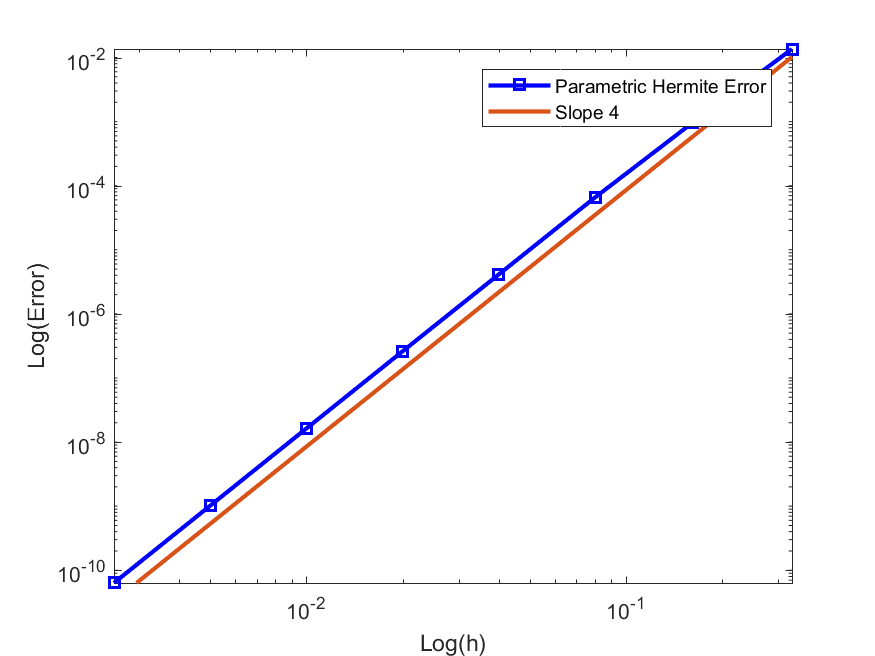}
\end{center}
\caption{ Spatial convergence  of particle path measured in the $L^{\infty}$ norm using parametric Hermite interpolation for example \ref{Ex3IC}).}
\label{HermSpatial}
\end{figure} 
\end{example}
The scheme used here is therefore fourth order in both space and time.  High-order spatial methods such as the area-preserving method of \cite{mcgregor2019area} or the sixth order curvature matching method of \cite{DeBoore} can also be applied here. A further discussion of these can be found in Section \ref{Results}.

Example \ref{Ex3} shows us that the idea of applying parametric interpolation to the extended characteristic equations (\ref{CharEqExt}) indeed allows us to construct high-order numerical approximations of  (\ref{ParNonHom}). The next step is showing how the constructed parametric interpolants can be used to correctly predict shock formation and initial shock location. 

The first piece of the puzzle is shock detection, which turns out to be very simple. Using the definition of shock formation discussed in Section \ref{Intro}, if at any point in the parametrization $s$ we have $x_s(s,t)<0$, then we know a shock has formed. Similarly we can check this on our parametric interpolants. If at time $t$ we have $\tilde{x}_{i_s}(s,t)>0$, but at time $t+\Delta t$ we have $\tilde{x}_{i_s}(s,t+\Delta t)<0$ within some interpolant $i$ at some point $s$ within its parametrization, then we predict a shock has formed between time $t$ and $t+\Delta t$. If we apply the equal-area principle on our curve $\tilde{x}(s,t+\Delta t)$, where $\langle \tilde{x}(s,t+\Delta t) , \tilde{u}(s,t+\Delta t) \rangle$ is given by parametric interpolation of the extended characterstic equations \ref{CharEqExt}, then we know by Theorem \ref{NonHomEqualArea} that the shock position is incorrect.  To overcome this we employ a modified equal-area principle. Recalling equations (\ref{ParNonHom}), we have that
\begin{align*}
x(s,t+\Delta t)&=x(s,t)+\int_t^{t+\Delta t}F'(u(s,\tau))\text{d}\tau, \quad \text{where $u$ is given by,}\\
u(s,t+\Delta t)&=u(s,t)+\int_t^{t+\Delta t}Q(u,x,\tau)\text{d}\tau.
\end{align*}
Instead of applying the equal-area principle on $\langle x(s,t+\Delta t), u(s,t+\Delta t)\rangle$, we fix the height of each particle, but still flow it under the correct horizontal motion given by solving the full system  for $\Delta t$ seconds. We therefore search for $s_1(t+\Delta t)$ and $s_2(t+\Delta t)$ such that
\begin{align}
&\int_{s_1(t+\Delta t)}^{s_2(t+\Delta t)}u(s,t)\frac{\partial}{\partial s}\left(x(s,t)+\int_{t}^{t+\Delta t}F'(u(s,\tau)\text{d}\tau\right)\text{d}=0, \quad \text{with,}\label{ModEA}\\
&x(s_1(t+\Delta t),t+\Delta t)=x(s_2(t+\Delta t),t+\Delta t).\nonumber
\end{align}
Once the shock position is found, we replace the overturned curve with a jump at $x(s_1(t+\Delta t ), t+\Delta t)$. Applying the vertical flow at this stage maps $u(s,t)\rightarrow u(s,t+\Delta t)$ which yields the left shock state $u(s_1(t+\Delta t), t+\Delta t)$ and right state $u(s_2(t+\Delta t), t+\Delta t)$.  It is important to note that the true shock is located somewhere in the multivalued region, say $x(s_1^*,t+\Delta t)=x(s_2^*,t+\Delta t)$, for some parameter values $s_1^*$ and $s_2^*$. The method of characteristics guarantees that the shock states must be given by $u(s_1^*,t+\Delta t)$ on the left and $u(s_2^*,t+\Delta t)$ on the right, where $\langle x(s_1^*,t+\Delta t), u(s_1^*,t+\Delta t) \rangle$ and $\langle x(s_2^*,t+\Delta t), u(s_2^*,t+\Delta t) \rangle$ are given by solving (\ref{CharEq}). Therefore, since our proposed method works directly with the characteristic equations the main source of error will come from the shock location itself.  The following Theorem proves that the one step error of (\ref{ModEA}) is second order accurate in time.

\begin{theorem}\label{ModEqualAreaThm}
Suppose $u(x,t)$ is a weak solution of (\ref{Cauchy}) with an isolated shock at position $x^*(t)$ with left state $u_L(t)$ and right state $u_R(t)$. Then, if at time $t$ the parametric curve given by solving (\ref{CharEq}) is an equal-area curve about $x^*(t)$, then the modified equal-area principle predicts the shock position at time $t+\Delta t$ with second order accuracy in time.
\end{theorem}

\begin{proof}
Suppose $\langle x(s,t) , u(s,t) \rangle$ is the parametric curve obtain by solving the characteristic equations (\ref{CharEq}) which contains an equal-area curve at position $x(s_1(t),t)=x(s_2(t),t)=x^*(t)$. This implies
\begin{equation}
\int_{s_1(t)}^{s_2(t)}u(s,t)x_s(s,t)\text{d}s=0.
\end{equation}
Applying the modified equal-area principle after $\Delta t$ seconds yields
\begin{align}
&\int_{s_1(t+\Delta t)}^{s_2(t+\Delta t)}u(s,t)x_s(s,t+\Delta t)\text{d}s=0, \quad \text{with,}\label{ModEA2}\\
&x(s_1(t+\Delta t),t+\Delta t)=x(s_2(t+\Delta t),t+\Delta t).\nonumber
\end{align}
Differentiating in $\Delta t$ yields

\begin{align*}
&u(s_2(t+\Delta t),t)x_s(s_2(t+\Delta t),t+\Delta t)s_2'(t+\Delta t)\\
&-u(s_1(t+\Delta t),t)x_s(s_1(t+\Delta t),t+\Delta t)s_1'(t+\Delta t)\\
&+\int_{s_1(t+\Delta t)}^{s_2(t+\Delta t)}u(s,t)x_{st}(s,t+\Delta t)\text{d}s=0.
\end{align*}
Applying the same techniques as in Theorem \ref{NonHomEqualArea} yields
\begin{align}
\frac{d}{dt}x(s_1(t+\Delta t),t+\Delta t)\left(u(s_2(t+\Delta t),t)-u(s_1(t+\Delta t),t)\right)\label{ModEAPF1}\\
+u(s_1(t+\Delta t),t)F'(u(s_1(t+\Delta t),t+\Delta t))\nonumber\\
-u(s_2(t+\Delta t),t)F'(u(s_2(t+\Delta t),t+\Delta t))\nonumber\\+\int_{s_1(t+\Delta t)}^{s_2(t+\Delta t)}u(s,t)F''(u(s,t+\Delta t))u_s(s,t+\Delta t)\text{d}s=0&.\nonumber
\end{align}
Applying integration by parts on the integration term we obtain
\begin{align}
\int_{s_1(t+\Delta t)}^{s_2(t+\Delta t)}u(s,t)F''(u(s,t+\Delta t))u_s(s,t+\Delta t)\text{d}s=u(s,t)F'(u(s,t+\Delta t))\biggr|_{s=s_1(t+\Delta t)}^{s_2(t+\Delta t)}&\label{ModEAPF2}\\
-\int_{s_1(t+\Delta t)}^{s_2(t+\Delta t)}u_s(s,t)F'(u(s,t+\Delta t))\text{d}s&\nonumber.
\end{align}
Plugging (\ref{ModEAPF2}) into (\ref{ModEAPF1}), cancelling terms and then solving for $\frac{d}{dt}x(s_1(t+\Delta t),t+\Delta t)$ yields the equation

\begin{equation}
\frac{d}{dt}x(s_1(t+\Delta t),t+\Delta t)=\frac{\int_{s_1(t+\Delta t)}^{s_2(t+\Delta t)}u_s(s,t)F'(u(s,t+\Delta t))\text{d}s}{u(s_2(t+\Delta t,t)-u(s_1(t+\Delta t,t)}.\label{ModEAPF3}
\end{equation}
Using the first order approximation in time of $F'(u(s,t+\Delta t))=F'(u(s,t))+\mathcal{O}(\Delta t)$ in (\ref{ModEAPF3}) allows us to compute the integration term  to first order, which gives us
\begin{equation}
\frac{d}{dt}x(s_1(t+\Delta t),t+\Delta t)=\frac{F(u(s_2(t+\Delta t,t))-F(u(s_1(t+\Delta t,t))}{u(s_2(t+\Delta t,t)-u(s_1(t+\Delta t,t)}+\mathcal{O}(\Delta t).\label{ModEAPF4}
\end{equation}
Using that $u(s(t+\Delta t),t)=u(s(t),t)+\mathcal{O}(\Delta t)$, long division yields 
\begin{equation}
\frac{d}{dt}x(s_1(t+\Delta t),t+\Delta t)=\frac{F(u(s_2(t),t))-F(u(s_1(t),t))}{u(s_2(t),t)-u(s_1(t),t)}+\mathcal{O}(\Delta t).\label{ModEAPF5}
\end{equation}
Integrating from $t$ to $t+\Delta t$ and using our assumptions that $x(s_1(t),t)=x^*(t)$ and $u(s_1(t),t)=u_L(t)$ and $u(s_2(t),t)=u_R(t)$, we obtain
\begin{equation}
x(s_1(t+\Delta t),t+\Delta t)=x^*(t)+\frac{F(u_R(t))-F(u_L(t))}{u_R(t)-u_L(t)}\Delta t+\mathcal{O}(\Delta t^2),\label{ModEAFinal}
\end{equation}
which agrees with the true shock position given by the Rankine-Hugoniot condition up to second order.
\end{proof}

\begin{remark}
Theorem \ref{ModEqualAreaThm} states that given the correct initial shock position along with upper and lower shock states, we are able to approximate the shock position after $\Delta t$ seconds with error proportional to $\Delta t^2$. Since our goal in this work is to obtain high-order numerical schemes, it would seem that applying the modified equal-area principle would eliminate the possibility of obtaining higher than second order. The idea here is that we can use an adapted time step of $\Delta \tilde{t}$ for a single step to obtain the desired accuracy to not impact the overall error in the problem, then continue with a different method once we have a sufficiently accurate initial shock position.

We also note we have access to the correct initial shock position and states at the moment the shock forms since the initial shock position and states are given from the characteristic equations (\ref{CharEq}). Therefore we will utilize the modified equal-area principle when we predict that a shock is forming.
\end{remark}

The final missing piece of our approach is how to propagate the shock once we know its initial position and shock states. As discussed above, our main objective is to predict the shock position as accurately as possible, since everything else comes directly from solving the characteristic equations, which we have already shown can be done to high accuracy. In the following subsection we introduce the parametric shock propagation method.

\subsection{Parametric Shock Propagation Method}\label{PSPM}
Suppose we have an isolated shock at $x^*(t)$ with smooth curves $u_L(x,t)$ to the left, defined for $x\leq x^*(t)$, and $u_R(x,t)$ to the right, defined for $x\geq x^*(t)$, with $u_L(x^*(t),t)>u_R(x^*(t),t)$, as depicted in Figure \ref{PSPMFig1}. Provided the shock remains isolated from other shocks or discontinuities, the Rankine-Hugoniot condition determines the shock motion through the equation
\begin{equation}
\frac{d}{dt}x^*(t)=\frac{F(u_L(x^*(t),t))-F(u_R(x^*(t),t))}{u_L(x^*(t),t)-u_R(x^*(t),t)},\label{RHC}
\end{equation}
where both $u_L(x,t)$ and $u_R(x,t)$ are flowing under the characteristic equations (\ref{CharEq}). In this situation it is clear that having an analytical representations for both $u_L(x,t)$ and $u_R(x,t)$ turns (\ref{RHC}) into an ordinary differential equation in which standard numerical methods can be applied.

\begin{figure}[!ht]
\begin{center}
\includegraphics[width=100mm,height=30mm]{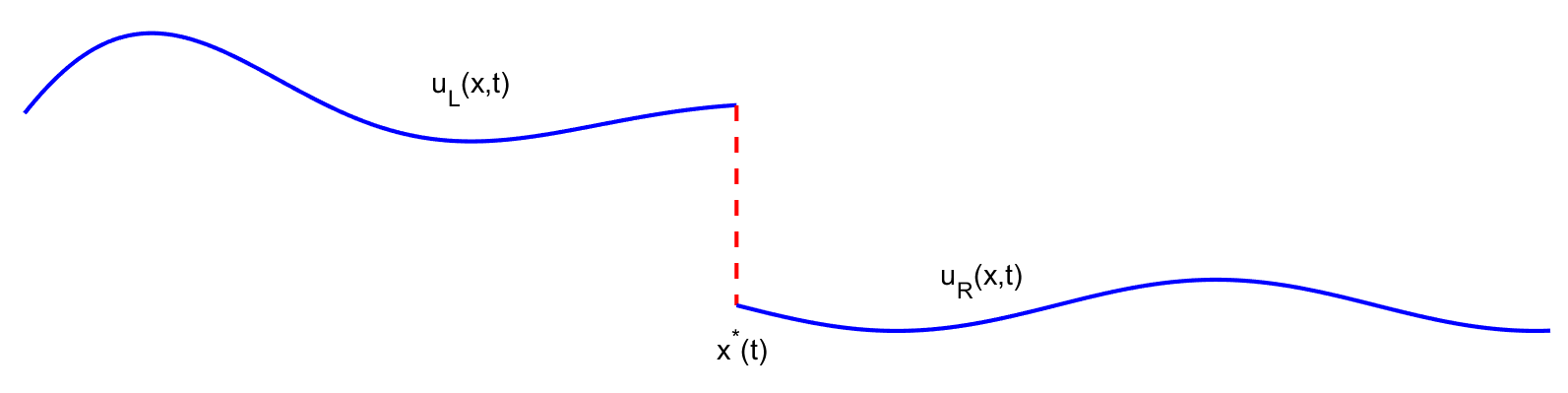}
\end{center}
\caption{ Initial configuration for Parametric Shock Propagation method. }
\label{PSPMFig1}
\end{figure} 

The key ingredient in allowing us to treat (\ref{RHC}) as a typical ordinary differential equation is splitting the problem into three distinct parts, the top curve, $u_L(x,t)$, the shock itself, located at $x^*(t)$, and the bottom curve $u_R(x,t)$. Since we are working with a uniformly convex flux function, we know that the shock will remain sandwiched between $u_L$ and $u_R$, provided the shock remains isolated.  Therefore, at time $t+\Delta t$ we can evaluate the slope field defined by (\ref{RHC}) at any point $z$ in the domain of both $u_L(x,t+\Delta t)$ and $u_R(x,t+\Delta t)$, where $u_L(z,t+\Delta t)>u_R(z,t+\Delta t)$, the region of overlap. From a numerical point of view this enables us to employ a wide range of methods, for example Runge-Kutta type methods, where the slope field is evaluated in several locations to obtain a high-order approximation of the shock speed between time $t$ and $t+\Delta t$. Later we see that multi-stage methods will require a stability condition to ensure that the slope field is always evaluated in a region of overlap.

To help clarify this approach we begin with a simple Euler step. Taylor expansion and then integration of (\ref{RHC}), as done in the proof of Theorem \ref{ModEqualAreaThm}, yields
\begin{equation}
x^*(t+\Delta t)=x^*(t)+\frac{F(u_L(x^*(t),t))-F(u_R(x^*(t),t))}{u_L(x^*(t),t)-u_R(x^*(t),t)}\Delta t +\mathcal{O}(\Delta t^2).\label{RHCT1}
\end{equation}
A single step of Forward Euler yields the predicted shock position at time $t+\Delta t$
\begin{equation}
\tilde{x}_E^*(t+\Delta t)=x^*(t)+\frac{F(u_L(x^*(t),t))-F(u_R(x^*(t),t))}{u_L(x^*(t),t)-u_R(x^*(t),t)}\Delta t\label{RHCEuler},
\end{equation}
which is equivalent to (\ref{RHCT1}) up to the second order term. An illustration is shown in Figure \ref{PSPMFig2}.

\begin{figure}[!ht]
\begin{center}
\includegraphics[width=100mm,height=35mm]{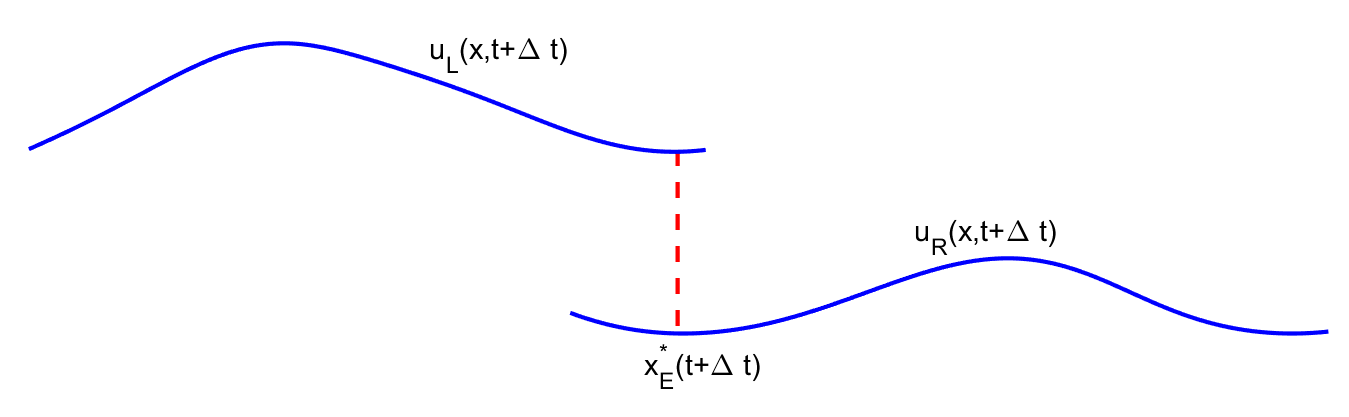}
\end{center}
\caption{ Predicted shock position using Forward Euler. }
\label{PSPMFig2}
\end{figure} 

Once we have verified that $\tilde{x}_E^*(t+\Delta t)$ is within the region of overlap, the portion of $u_L(x,t+\Delta t)$ to the right of $\tilde{x}_E^*(t+\Delta t)$ is removed and the portion of $u_R(x,t+\Delta t)$ to the left of $\tilde{x}_E^*(t+\Delta t)$ is removed. As shown in Figure \ref{PSPMFig3}, this process has brought us back to an equivalent state to Figure \ref{PSPMFig1}.

\begin{remark}
If at a time $t^*\in[t,t+\Delta t]$ either $u_L(x,t)$ or $u_R(x,t)$ becomes multi-valued, a fractional step of size $\Delta \tilde{t}$ must to taken instead to $\Delta t$ to locate the newly formed shock. Once located, we can propagate both shocks in the manner described above until either shock comes into contact with a discontinuity.
\end{remark}

\begin{figure}[!ht]
\begin{center}
\includegraphics[width=100mm,height=35mm]{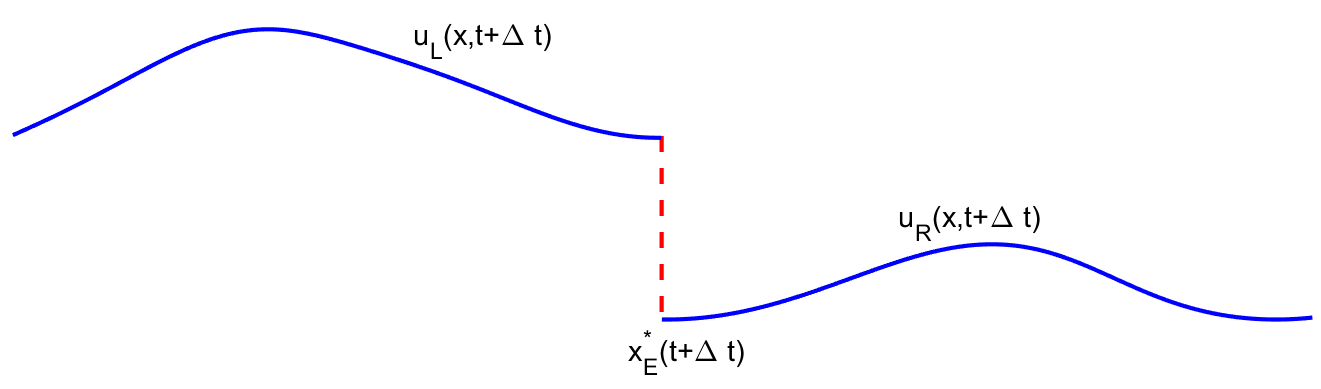}
\end{center}
\caption{ Predicted shock position using Forward Euler after removing overlap. }
\label{PSPMFig3}
\end{figure} 

This idea can easily be adapted to Improved Euler, also known as the predictor corrector method. As seen in Figure \ref{PSPMFig2}, we are able to now evaluate the slope at $\tilde{x}_E^*(t+\Delta t)$ using $u_L(x,t+\Delta t)$ and $u_R(x,t+\Delta t)$. Averaging the predicted slope, given by Forward Euler, and the corrected slope, given by this additional evaluation, we obtain the Improved Euler method 
\begin{align}
\tilde{x}^*(t+\Delta t)=x^*(t)+\frac{\Delta t}{2}\left(\frac{F(u_L(x^*(t),t))-F(u_R(x^*(t),t))}{u_L(x^*(t),t)-u_R(x^*(t),t)}\right.&\label{RHCIE}\\
+\left.\frac{F(u_L(\tilde{x}_E^*(t+\Delta t),t+\Delta t))-F(u_R(\tilde{x}_E^*(t+\Delta t),t+\Delta t))}{u_L(\tilde{x}_E^*(t+\Delta t),t+\Delta t)-u_R(\tilde{x}_E^*(t+\Delta t),t+\Delta t)} \right).\nonumber
\end{align}
Although we omit the calculation due to its length, by setting $\tilde{x}_E^*(t+\Delta t)=x^*(t)+\frac{F(u_L(x^*(t),t))-F(u_R(x^*(t),t))}{u_L(x^*(t),t)-u_R(x^*(t),t)}\Delta t$ and Taylor expanding (\ref{RHCIE}) in $\Delta t$ we obtain exactly the Taylor expansion of $x^*(t+\Delta t)$ from equation (\ref{RHC}) up to the $\Delta t^3$ term. In Section \ref{Results} we show how these ideas generalize to higher-order Runge-Kutta methods. Using computational software we have confirmed that indeed the Taylor expansions cancel up to the $\Delta t^5$ term in the Runge-Kutta 4 case, but these details are omitted due to their length.

Before moving to the results section we first need to justify that numerically we can always find a $\Delta t$ small enough such that each evaluation of multi-stage Runge-Kutta methods land in the region of overlap. But first we need a rigorous definition of the region of overlap.

\begin{definition}\label{ROO}
Suppose we have a shock at $x^*(t)$ with smooth curves $\tilde{u}_L(x,t)$ to the left, defined for $x\leq x^*(t)$, and $\tilde{u}_R(x,t)$ to the right, defined for $x\geq x^*(t)$, with $\tilde{u}_L(x^*(t),t)>\tilde{u}_R(x^*(t),t)$ as in Figure \ref{PSPMFig1}, where $\tilde{u}_L$ and $\tilde{u}_R$ are obtained by numerically solving the system (\ref{CharEq}). Without loss of generality suppose we parametrize $\tilde{u}_L(x,t)$ by the curve $\langle x_L(s,t), u_L(s,t) \rangle$, where $x_L(1,t)=x^*(t)$ and $u_L(1,t)=\tilde{u}_L(x^*(t),t)$. Similarly we parametrize $\tilde{u}_R(x,t)$ by the curve $\langle x_R(s,t), u_R(s,t) \rangle$, where $x_R(0,t)=x^*(t)$ and $u_R(0,t)=\tilde{u}_R(x^*(t),t)$. The Region of Overlap at time $t+\Delta t$ is all $x\in\mathbb{R}$ such that
$ x_R(0,t+\Delta t)<x<x_L(1,t+\Delta t)$.
\end{definition}

%

\begin{theorem}\label{Overlap}
Suppose we have an isolated shock at $x^*(t)$ with smooth curves $\tilde{u}_L(x,t)$ to the left, defined for $x\leq x^*(t)$, and $\tilde{u}_R(x,t)$ to the right, defined for $x\geq x^*(t)$, with $\tilde{u}_L(x^*(t),t)>\tilde{u}_R(x^*(t),t)$ as in Figure \ref{PSPMFig1}, where $\tilde{u}_L$ and $\tilde{u}_R$ are obtained by numerically solving the system (\ref{CharEq}). In addition we suppose that neither curve $u_L(x,t)$ or $u_R(x,t)$ form an additional shock between times $t$ and $t+\Delta t$. Then, given a Runge-Kutta method, there exists a $\Delta t$ small enough such that each stage of the method can be evaluated from the slope field defined by (\ref{RHC}), and therefore the utilized Runge-Kutta method can be constructed as in the standard ordinary differential equation setting.
\end{theorem}
\begin{proof}
As can be found in numerous numerical analysis of ordinary differential equations books, for example \cite{RK}, Runge-Kutta methods obtain the numerical approximation $\tilde{x}(t+\Delta t)$ through a convex combination of slopes in a neighbourhood of $\tilde{x}(t)$. Each slope, $k_i$, has the same first term, 
\begin{equation}
k_i=\frac{F(u_L(\tilde{x}(t),t))-F(u_R(\tilde{x}(t),t))}{u_L(\tilde{x}(t),t)-u_R(\tilde{x}(t),t)}+\mathcal{O}(\Delta t).
\end{equation}
Uniform convexity of the flux function $F$ implies that $F'(u_1)<F'(u_2)$ whenever $u_1<u_2$, therefore  since condition (\ref{RHC}) is the average value of $F'(u)$ between $u_R$ and $u_L$ we have
\begin{equation}
F'(u_R(\tilde{x}(t),t))<\frac{F(u_L(\tilde{x}(t),t))-F(u_R(\tilde{x}(t),t))}{u_L(\tilde{x}(t),t)-u_R(\tilde{x}(t),t)}<F'(u_L(\tilde{x}(t),t)).\label{UniformConv}
\end{equation} 
The region of overlap, given by Definition \ref{ROO}, has left boundary given by $x_R(0,t+\Delta t)=x_R(0,t)+F'(u_R(\tilde{x}(t),t))\Delta t+\mathcal{O}(\Delta t^2)$ and right boundary $x_L(1,t+\Delta t)=x_L(1,t)+F'(u_L(\tilde{x}(t),t))\Delta t+\mathcal{O}(\Delta t^2)$. Therefore inequality (\ref{UniformConv}) implies,  for small enough $\Delta t$, we have
\begin{equation}
\frac{d}{d \Delta t}x_R(0,t+\Delta t)<k_i<\frac{d}{d \Delta t}x_L(1,t+\Delta t).
\end{equation}
Therefore, provided $\Delta t$ is small enough, every stage of a Runge-Kutta method can be evaluated from the slope field given by the numerical approximation of $u_L(x,t)$ and $u_R(x,t)$ in the differential equation (\ref{RHC}).
\end{proof}

\begin{remark}
It is important to note that either $u_L(x,t)$ or $u_R(x,t)$ may become multi-valued between time $t$ and $t+\Delta t$. If this occurs within the region of overlap then the resulting slope field given by (\ref{RHC}) is no longer reliably smooth. A smaller time step is required and then a point of contact between the two shocks must be approximated.
\end{remark}

In the following section we present detailed numerical examples showing that the Parametric Shock Propagation Method captures the shock position to high spatial and temporal order.

\section{Numerical results}\label{Results}

\begin{example}\label{Ex4}


In this example we consider the following Cauchy problem
\begin{equation}
\begin{cases}
u_t+uu_x=-u(1-u)\label{IntersectEx}\\
u(x,0)=g(x),
\end{cases}
\end{equation}
where,
\begin{equation*}g(x)=
\begin{cases}
0.9, \quad \text{for $x<2$}\\
0.5, \quad \text{for $2<x<2.5$}\\
0.2, \quad \text{for $x>2.5.$}
\end{cases}
\end{equation*} 
\begin{figure}[!ht]
\begin{center}
\includegraphics[width=70mm,height=40mm]{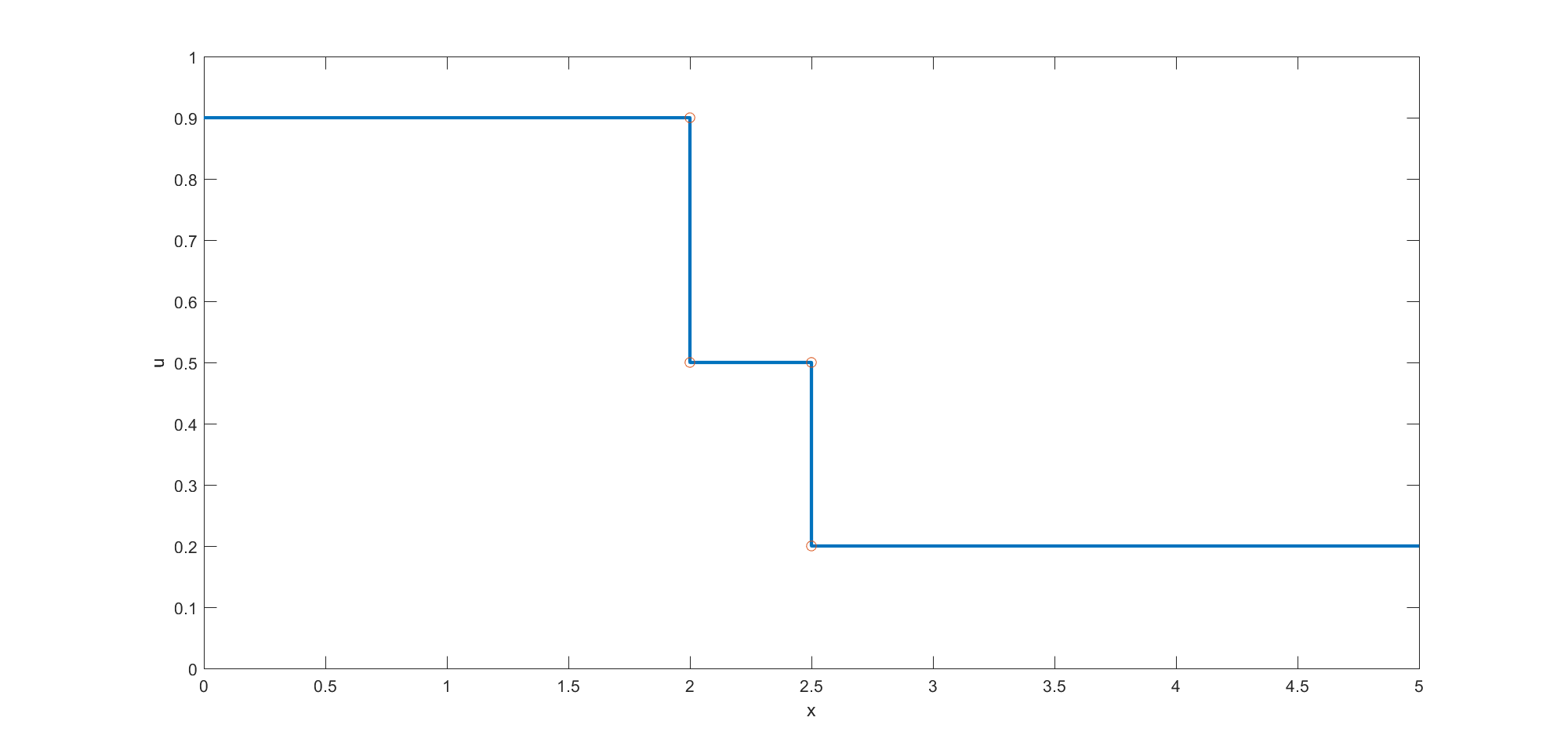}
\end{center}
\caption{ Initial condition from Example \ref{Ex4}. }
\label{ShockCollisionIC}
\end{figure} 
The characteristic equations associated with each constant state of $g(x)$ can be solved exactly. The equation for $u$ can be solved independently, yielding
\begin{equation}
u(x_0,t)=\frac{1}{1+(\frac{1}{g(x_0)}-1)e^{t}}.\label{ShockCollisionU}
\end{equation}
Using (\ref{ShockCollisionU}) we obtain the height of each state at time $t$, given by
\begin{align}
u_L(t)&=\frac{1}{1+(\frac{1}{0.9}-1)e^{t}}=\frac{1}{1+(\frac{0.1}{0.9})e^{t}}\label{ShockCollisionUL}\\
u_M(t)&=\frac{1}{1+(\frac{1}{0.5}-1)e^{t}}=\frac{1}{1+e^{t}}\label{ShockCollisionUM}\\
u_R(t)&=\frac{1}{1+(\frac{1}{0.2}-1)e^{t}}=\frac{1}{1+4e^{t}}\label{ShockCollisionUR}
\end{align}
Equations (\ref{ShockCollisionUL}-\ref{ShockCollisionUR}) allow us to compute the shock speeds corresponding the discontinuities in $g(x)$ at $x=2$ and $x=2.5$. The shock starting at $x=2$ must satisfy the equation
\begin{equation}
\frac{d}{dt}x_1^*(t)=\frac{F(u_L(t))+F(u_M(t))}{u_L(t)+u_M(t)}=\frac{u_L(t)+u_M(t)}{2}.
\end{equation}
Therefore we have 
\begin{align}
x_1^*(t)&=2+\frac{1}{2}\int_0^t\frac{1}{1+(\frac{0.1}{0.9})e^{\tau}}+\frac{1}{1+e^{\tau}}\text{d}\tau.\nonumber\\
&=2+t+\frac{1}{2}\left(\ln\left(\frac{2}{0.9}\right)-\ln\left(\left(1+\frac{0.1}{0.9}e^t\right)\left(1+e^t\right)\right)\right)\label{ShockCollisionLeftShock}.
\end{align}
Similarly we obtain an equation for the shock starting at $x=2.5$,
\begin{align}
x_2^*(t)&=2.5+t+\frac{1}{2}\left(\ln\left(10\right)-\ln\left(\left(1+e^t\right)\left(1+4e^t\right)\right)\right)\label{ShockCollisionRightShock}.
\end{align}

For this initial condition the two shocks collide at time $t^*$ given by
\begin{equation}
t^*=\ln\left(\frac{(1+\frac{0.1}{0.9})-5e}{\frac{0.5}{0.9}e-\frac{4}{0.9}}\right)\approx 1.44769.
\end{equation}
Figure \ref{ShockCollisionTime} shows the solution before and after the shocks intersect.

\begin{figure}[!ht]
\begin{center}
\includegraphics[width=45mm,height=40mm]{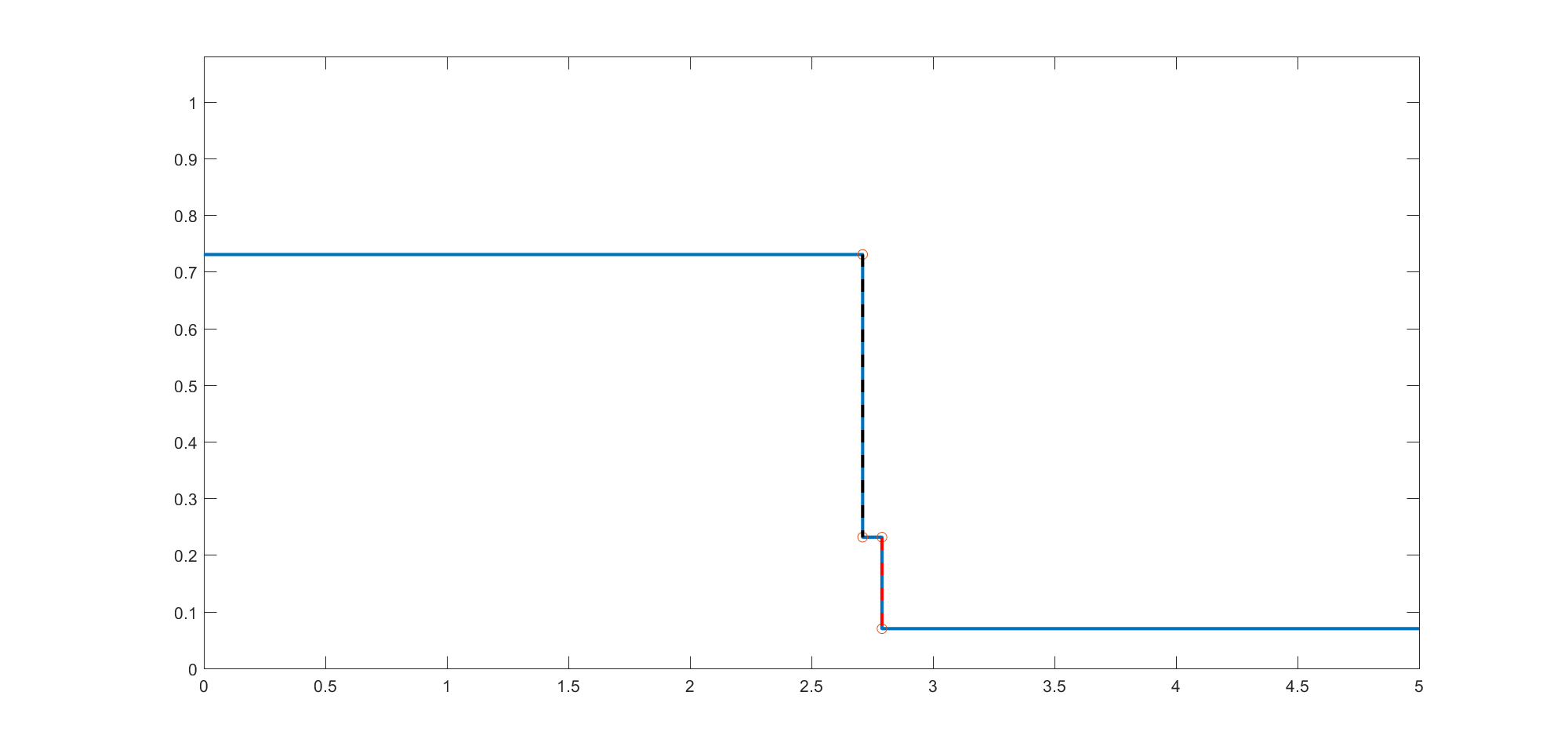}
\includegraphics[width=45mm,height=40mm]{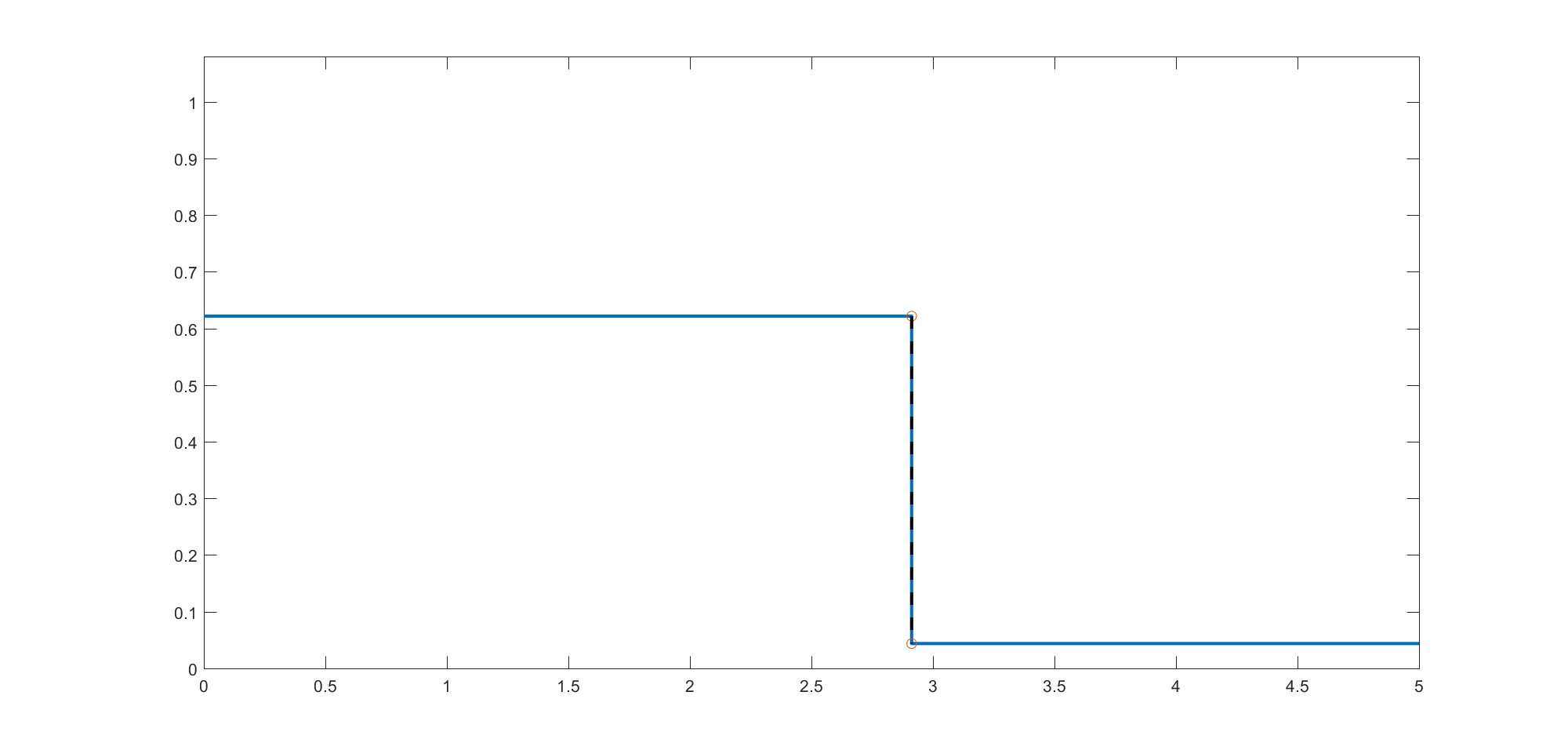}
\end{center}
\caption{ The solution from Example \ref{Ex4} before and after the shocks collide at time $t^*\approx 1.44769$. }
\label{ShockCollisionTime}
\end{figure} 

Applying the parametric shock propagation method we compute the approximate position of the shock at time $t=2$. Since we are working with constant states, there is no spatial interpolation error. Therefore the only source of error comes from the propagation of the shocks and their collision at $t^*$. Applying the fourth order Runge-Kutta version of the method described in Section \ref{PSPM} we obtain fourth order accuracy in the shock position. The convergence plot is shown in Figure \ref{ShockCollisionTimeer}.

\begin{figure}[!ht]
\begin{center}
\includegraphics[width=70mm,height=50mm]{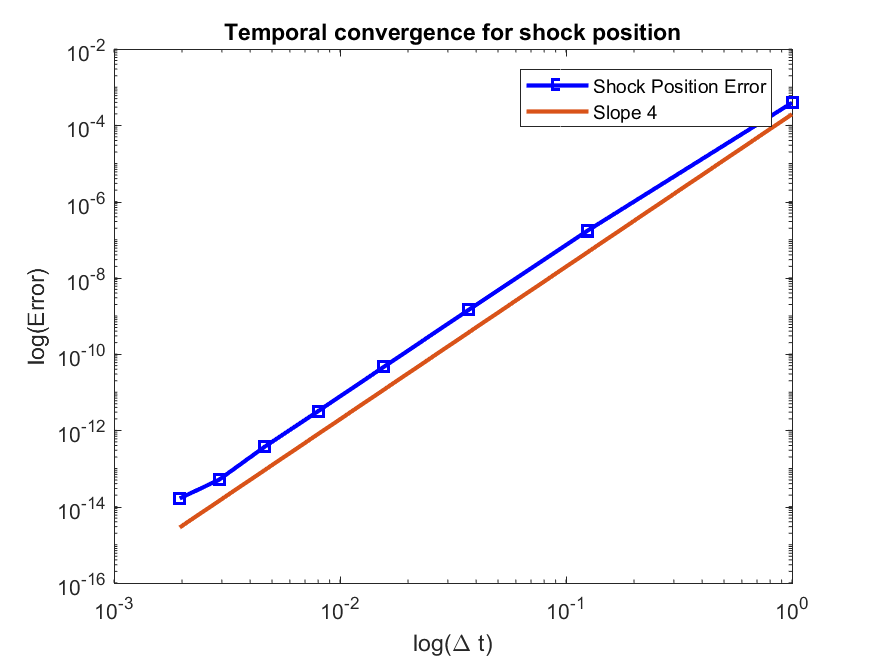}
\end{center}
\caption{ Convergence of the fourth order Runge-Kutta Shock Propagation Method. }
\label{ShockCollisionTimeer}
\end{figure}
\end{example}

\begin{remark}
In order to not lose any accuracy as the shocks merge we are required to numerically approximate the time $t^{**}$ when the shocks collide. In the above example we did this by taking a few adapted time steps $\Delta \tilde{t}$ until the shocks were within machine precision of each other and thus merged into a single shock. Ensuring that each stage of the Runge-Kutta method evaluates the slope field correctly is required. In this case we had to ensure that the propagation of the left shock did not require any evaluations past the right shock. This condition can easily be checked at each stage of the Runge-Kutta method.
\end{remark}

\begin{example}\label{Ex5}
In our final example we aim to track the shock propagating under the boundary value problem
\begin{equation}
\begin{cases}
u_t+uu_x=\sin(x)u\label{Ex5IC}\\
u(0,t)=\frac{1}{2}, \quad \text{for $t\geq 0$}\nonumber\\
u(x,0)=0, \quad \text{for $x>0$}\nonumber.
\end{cases}
\end{equation}
This example breaks down into a Riemann problem where the left shock state is given by the solution to the curve traced out by the system
\begin{equation}\label{Ex5IC2}
 \left\{
\begin{array}{l}
     \dot{x}=u\\
     \dot{u}=\sin(x)u\\
     x(0)=0\\
     u(0)=\frac{1}{2}.
\end{array}\right.
\end{equation}
As we saw in Example \ref{Ex3} we are able to perform a high-order spatial and temporal interpolation of this curve using parametric interpolation methods. Before jumping into the full numerical method we first ensure that the appropriate temporal error from the shock propagation method is achieved. Using the analytical solution to (\ref{Ex5IC}), derived in Example \ref{Ex3}, we simply solve the differential equation given by the Rankine-Hugoniot Condition,
\begin{equation}
\frac{d}{dt}x^*(t)=\frac{3-2\cos(x^*(t))}{4}.
\end{equation}
The true shock position is therefore 
\begin{equation}
x(t)=\arctan\left(\frac{\tan\left(\frac{\sqrt{5}}{4}t\right)}{\sqrt{5}}\right)\label{Ex5ShockPos}.
\end{equation}
It comes to no surprise that the shock propagation method succeeds to obtain fourth-order accuracy when employing the fourth order Runge-Kutta method on the differential equation (\ref{Ex5ShockPos}). The convergence of the shock position at time $t=5$ is plotted in Figure \ref{ShockTempConv}.

\begin{figure}[!ht]
\begin{center}
\includegraphics[width=70mm,height=50mm]{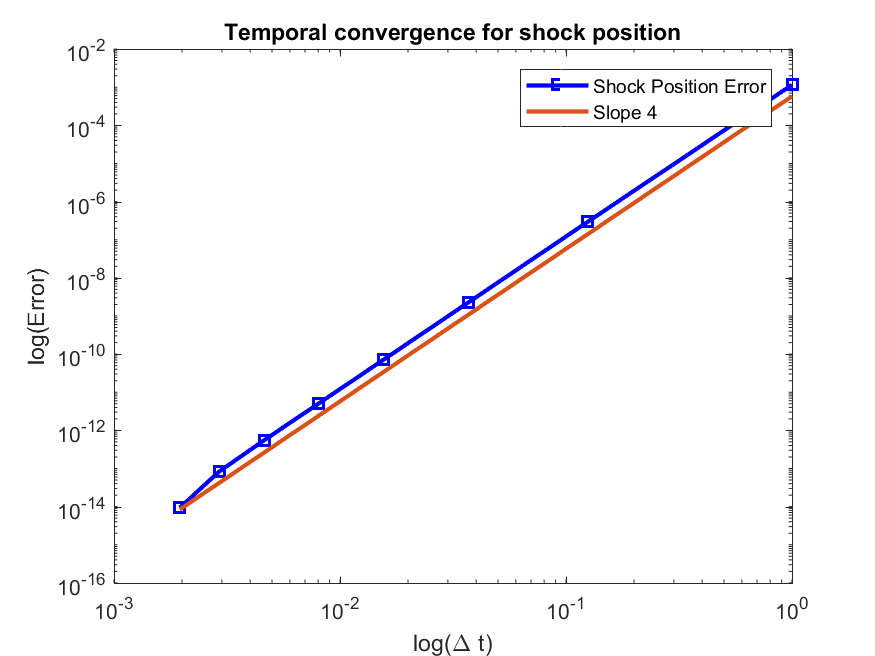}
\end{center}
\caption{ Convergence of the fourth order Runge-Kutta Shock Propagation Method on exact solution. }
\label{ShockTempConv}
\end{figure}

Now we solve the same the problem utilizing the parametric interpolation discussed in Section \ref{PIF}. To do this we compute the extended characteristic equations as done in Example \ref{Ex3}. Then, at each stage of the Runge-Kutta shock propagation we evaluate the approximate value of $u(x^*(t+\Delta t))$ using the computed parametric interpolants. A plot of the solution at time $t=3.5$ and $t=5$ is given in Figure \ref{ShockPlot}.

\begin{figure}[!ht]
\begin{center}
\includegraphics[width=45mm,height=40mm]{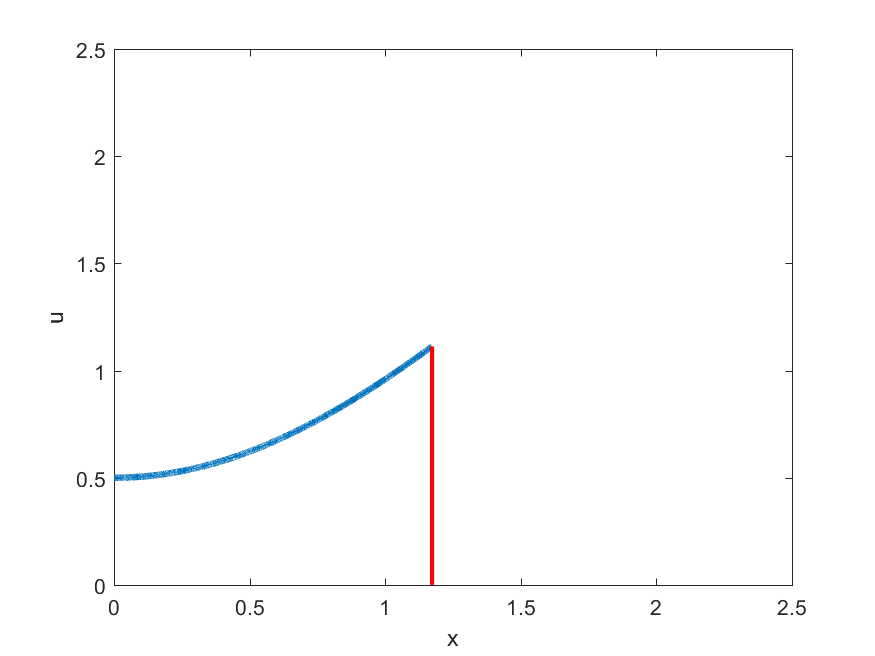}
\includegraphics[width=45mm,height=40mm]{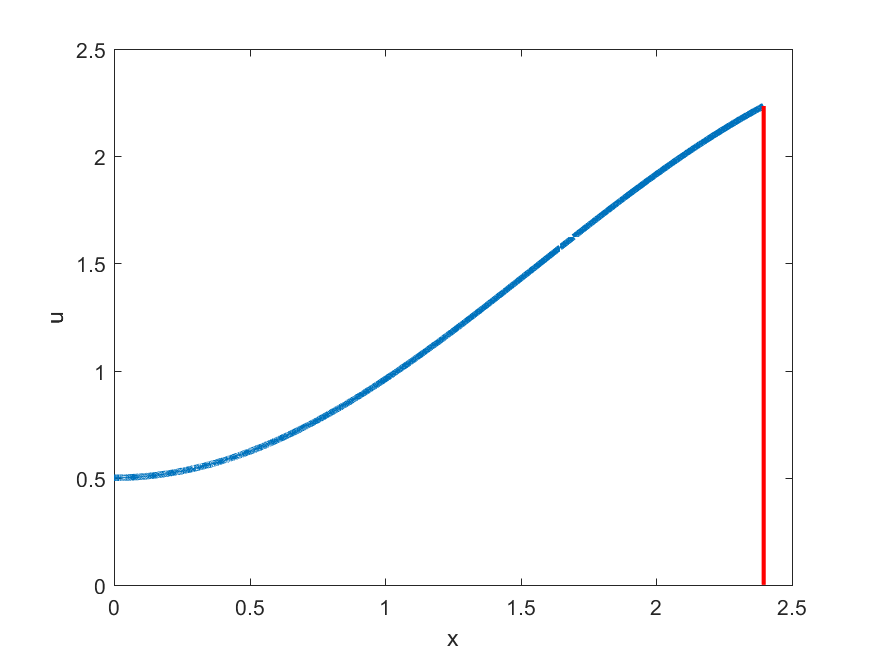}
\end{center}
\caption{ The solution from Example \ref{Ex5} at times $t=3.5$ and $t=5$. }
\label{ShockPlot}
\end{figure} 

Taking small enough time steps to ensure that the temporal error does not influence the convergence, we compute the spatial convergence in the shock position at time $t=5$. First we employ a basic parametric Hermite cubic interpolation and obtain the fourth order convergence seen in Figure \ref{ShockSpatialConv}. Next, using a sequence of Hermite interpolants to compute the approximate area, we compute the area-preserving cubic B\'ezier of \cite{mcgregor2019area} and indeed achieve fifth order as predicted. 

\begin{figure}[!ht]
\begin{center}
\includegraphics[width=70mm,height=50mm]{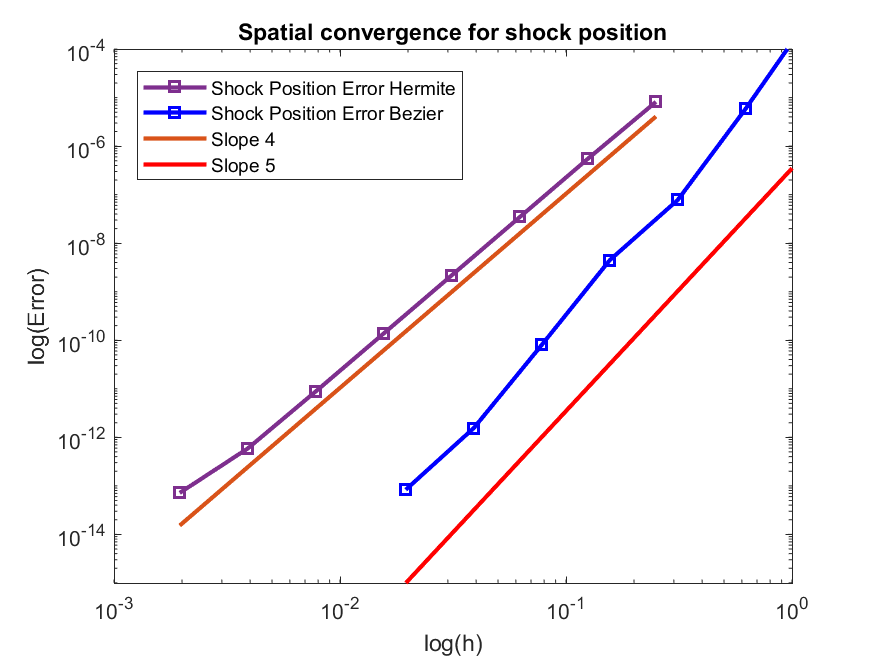}
\end{center}
\caption{ Convergence of the fourth order Runge-Kutta Parametric Shock Propagation Method.}
\label{ShockSpatialConv}
\end{figure}

\end{example}

\begin{remark}
Unlike the homogeneous case, it can be difficult to compute the area under each portion of the parametric curve given by solving the characteristic equations. In Example \ref{Ex5} we used approximately ten Hermite interpolants to approximate the area under each B\'ezier curve. In the end we obtained one order higher, but with an added computational cost. The main takeaway from this example is that parametric interpolation methods which are capable of obtaining high-order can indeed be applied to this framework resulting in high-order numerical approximation of the weak solution, even in the presence of spatial and solution dependent source terms.
\end{remark}

\section{Discussion}\label{Discussion}
In this paper we present a framework for obtaining high-order numerical solutions of scalar conservation laws in one-space dimension containing shocks. Unlike the majority of methods available, the approach presented here, which relies on high-order parametric interpolation of the characteristic equations, preserves the spatial and temporal order at the shock location. In the homogeneous case, shocks are located through an equal-area projection, while in the non-homogeneous case,  a modified  equal-area principle is applied for a single step to find an initial shock location to second order accuracy. The Rankine-Hugoniot condition can then be solved directly using the left and right shock states given by parametric interpolation of the characteristic equations. 

In the homogeneous case we showed that the data required to perform the area-preserving B\'ezier interpolation of \cite{mcgregor2019area} is easily obtained from the characteristic equations. Fifth-order spatial accuracy, or better, is indeed obtained as expected. In the non-homogeneous case we showed that a simple parametric Hermite interpolation works very well, and although we succeeded to obtain the spatial fifth-order accuracy using the area-preserving interpolation, it came at an additional computational cost. If one seeks higher-order spatial interpolation we suggest doing a further extension of the characteristic equations to include second derivatives in the parametrization. From here one would be able to apply the sixth-order accurate curvature matching method discussed in \cite{DeBoore}.
There are also many parametric quintic methods within the parametric interpolation literature that could be employed.
Going forward we are interested in adapting these methods to system of conservation laws in one-space dimension. In particular we are interested in studying system of two equations containing Riemann Invariants. We also are interested in transforming the methods discussed in this paper onto a grid, instead of a tracking particles as they flow along the characteristics. From this perspective it would be interesting to study the differences between area-preserving parametric methods and finite volume methods.  
\bibliographystyle{siam}
\bibliography{references}

\end{document}